\documentclass{article}%
\usepackage{amssymb}
\usepackage{amsmath}
\usepackage{amsfonts}
\usepackage{graphicx}%
\newtheorem{theorem}{Theorem}

\newtheorem{corollary}[theorem]{Corollary}

\newtheorem{definition}[theorem]{Definition}
\newtheorem{example}[theorem]{Example}

\newtheorem{lemma}[theorem]{Lemma}

\newtheorem{proposition}[theorem]{Proposition}
\newtheorem{remark}[theorem]{Remark}

\newenvironment{proof}[1][Proof]{\noindent\textbf{#1.} }{\ \rule{0.5em}{0.5em}}
\begin{document}

\title{Instantons in $G_{2}$ manifolds from $J$-holomorphic curves in coassociative
submanifolds }
\author{Naichung Conan Leung, Xiaowei Wang and Ke Zhu}
\maketitle

\begin{abstract}
In $G_{2}$ manifolds, $3$-dimensional associative submanifolds (instantons)
play a role similar to $J$-holomorphic curves in symplectic geometry. In
\cite{LWZ1}, instantons in $G_{2}$ manifolds were constructed from regular
$J$-holomorphic curves in coassociative submanifolds. In this exposition
paper, after reviewing the background of $G_{2}$ geometry, we explain the main
ingredients in the proofs in \cite{LWZ1}. We also construct new examples of instantons.

\end{abstract}

\section{Introduction}

\bigskip A $G_{2}$-manifold $M$ is a Riemannian manifold of dimension seven
equipped with a nontrivial parallel $2$-fold \textit{vector cross product
}(VCP) $\times$. If the VCP is $1$-fold instead, namely a Hermitian complex
structure, then the manifold is a K\"{a}hler manifold $X$.\footnote{Throughout
this paper $M$ is a $G_{2}$-manifold and $X$ is a symplectic manifold possibly
with extra structures, like K\"{a}hler or CY.} Note that every $2$-fold VCP
comes from the restriction of the algebra product on the octonions
$\mathbb{O}$ or quaternions $\mathbb{H}$ to its imaginary part. Thus the only
submanifolds in $M$ preserved by $\times$ are of dimension three and they are
called \textit{instantons} or \textit{associative submanifolds}, which are
analogs to holomorphic curves in K\"{a}hler manifolds.

In Physics, $G_{2}$-manifolds are internal spaces for compactification in
M-theory in eleven dimensional spacetimes, similar to the role of
\textit{Calabi-Yau threefolds} $X$ in string theory in ten dimensional
spacetimes. Instantons in string theory are holomorphic curves $\Sigma$ in $X$
with the natural boundary condition being $\partial\Sigma$ lies inside a
Lagrangian submanifold $L$ in $X$. Very roughly speaking, the Fukaya category
$Fuk\left(  X\right)  $ is defined by counting holomorphic disks with
Lagrangian boundary condition. The homological mirror symmetry (HMS)
conjecture of Kontsevich says that $Fuk\left(  X\right)  $ is equivalent to
the derived category $D^{b}\left(  X^{\vee}\right)  $ of coherent sheaves of
the mirror manifold $X^{\vee}$. The proof of this conjecture and its
generalizations in many specific cases relies on the work of Fukaya-Oh
\cite{Fukaya Oh} when $X=T^{\ast}L$ and the Lagrangian $L_{t}$ is the graph of
an exact one form $df$ scaled by small $t\in\mathbb{R}$. They showed that
holomorphic disks with boundary in $L\cup L_{t}$ one to one correspond to
gradient flow lines of $f$, which are in fact instantons in quantum mechanics
according to Witten's Morse theory \cite{Wi Morse SUSY}. (Fukaya-Oh
\cite{Fukaya Oh} actually proved the case of $k$ Lagrangians $L_{t}%
^{j}=tdf_{j}$ ($j=1,2,\cdots k$), where holomorphic polygons bounding on these
Lagrangians correspond to gradient flow trees of the Morse functions $\left\{
f_{j}\right\}  _{j=1}^{k}$.)

In \cite{LWZ1} we proved a corresponding result for instantons in any $G_{2}%
$-manifold with boundary in the coassociative submanifold $C\cup C_{t}$. Here
the family $C_{t}$ is constructed by a self-dual harmonic two form $\omega$ on
the four manifold $C$. We assume that $\omega$ is non-degenerate, thus
defining an almost complex structure $J$ on $C$. Our main result Theorem
\ref{correspondence} gives a correspondence between such instantons in $M$ and
$J$-holomorphic curves $\Sigma$ in $C$. Thus the number of instantons in $M$
is related to the Seiberg-Witten invariant of $C$ by the celebrated work of
Taubes on GW=SW (\cite{Ta Gr Sw}, \cite{Ta SW equal GW}, \cite{Ta SW GW}). We
suspect this holds true without the need of the non-degeneracy of $\omega$.

Our result is similar to the $k=2$ case of Fukaya-Oh \cite{Fukaya Oh}, however
the analysis involved in the proof is essentially different from theirs in the
following 3 aspects.

\begin{enumerate}
\item The instanton equation is on $3$\emph{-dimensional} domains, and there
is no analogous way of finding associative submanifold by constructing\emph{
associative maps }as was done in constructing $J$-holomorphic curves thanks to
the\emph{ conformality }of Cauchy-Riemann equation. So we have to deform
\emph{submanifolds} rather than \emph{maps} as in the Lagrangian Floer theory.
A good choice of normal frames of submanifolds turns out to be essential.

\item The instanton equation is \emph{more nonlinear} than the Cauchy-Riemann
equation for $J$-holomorphic curves. It is a first order PDE system involving
cubic terms $\frac{\partial V^{i}}{\partial x^{1}}\frac{\partial V^{j}%
}{\partial x^{2}}\frac{\partial V^{k}}{\partial x^{3}}$, while the
Cauchy-Riemann operator $\frac{\partial u}{\partial\tau}+J\left(  u\right)
\frac{\partial u}{\partial t}$ has no product of derivative terms.
Consequently, the needed \emph{quadratic estimate} appears to be unavailable
in the $W^{1,p}$ setting, as for $V^{i},V^{j},V^{k}\in W^{1,p}$,
$\frac{\partial V^{i}}{\partial x^{1}}\frac{\partial V^{j}}{\partial x^{2}%
}\frac{\partial V^{k}}{\partial x^{3}}\notin L^{p}$ in general. So instead we
use the \emph{Schauder }($C^{1,\alpha}$) \emph{setting}, which creates new
complications (Subsection \ref{D-inverse}).

\item The linearized instanton equation is \emph{more weakly coupled} than the
Cauchy-Riemann equation. It is a Dirac type equation for spinor $\left(
u,v\right)  \in\mathbb{S}^{+}\oplus\mathbb{S}^{-}$ where $u$ and $v$ play the
role of the real and imaginary parts in Cauchy-Riemann equations, but the
interrelation between $\nabla u$ and $\nabla v$ becomes weaker. This causes
several difficulties in the $C^{1,\alpha}$ estimates (See comments below
$\left(  \ref{CR1}\right)  $).
\end{enumerate}

Besides above difficulties, our domains are $\left[  0,\varepsilon\right]
\times\Sigma$ for compact Riemannian surfaces $\Sigma$, and they
\emph{collapse} to $\Sigma$ as $\varepsilon\rightarrow0$, causing \emph{lack
of uniform ellipticity,} which in turn creates difficulty to obtain a uniform
right inverse bound needed in gluing arguments. This also occurs in
Proposition 6.1 of Fukaya-Oh \cite{Fukaya Oh} in the $W^{1,p}$ setting, but in
our $C^{1,\alpha}$ setting the boundary estimates become more subtle.

To deal with these difficulties, our paper \cite{LWZ1} becomes rather
technical. Therefore in this article we \ give an outline of the main
arguments. The organization is as follows. In Section \ref{G2}, we review the
background of $G_{2}$ geometry, instantons, coassociative boundary condition
and give the motivations of counting instantons. In Section
\ref{thin-instanton} we state the main theorem in \cite{LWZ1} and explain the
main ingredients in the proof. In Section \ref{application} we apply the
theorem to construct new examples of instantons and discuss possible generalizations.

\textbf{Acknowledgement.} We thank S.-T. Yau for comments on our paper
\cite{LWZ1}, and Baosen Wu for useful discussions on complex surfaces. The
third author thanks Clifford Taubes for interest and support. The work of N.C.
Leung was substantially supported by a grant from the Research Grants Council
of the Hong Kong Special Administrative Region, China (Project No.
CUHK401411). The work of K. Zhu was partially supported by the grant of
Clifford Taubes from the National Science Foundation.

\section{$\bigskip$Review on $G_{2}$-geometry\label{G2}}

\subsection{$G_{2}$ manifolds}

$G_{2}$ manifolds are $7$-dimensional Riemannian manifolds $\left(
M,g\right)  $ with a parallel\emph{\ vector cross product} (VCP) $\times$,
i.e. for $u,v\in T_{p}M$ there is a VCP $\times$ such that
\[%
\begin{array}
[c]{cl}%
\text{(i) } & u\times v\,\text{\ is perpendicular to both }u\text{ and }v.\\
\text{(ii) } & \left\vert u\times v\right\vert =\left\vert u\wedge
v\right\vert =\text{Area of the parallelogram spanned by }u\text{ and }v,
\end{array}
\]
and $\times$ is invariant under the parallel transport of the Levi-Civita
connection $\nabla$.

In the octonion algebra $\mathbb{O}$, we can construct a VCP $\times$ on
$\operatorname{Im}$ $\mathbb{O}$ as follows:
\begin{equation}
u\times v=\operatorname{Im}\left(  \overline{v}u\right)  ,
\label{cross-product}%
\end{equation}
where $\overline{v}$ is the conjugate of $v$. The same formula for
$\mathbb{H}$ gives another VCP on $\operatorname{Im}\mathbb{H}$, which is
indeed the useful vector product on $\mathbb{R}^{3}$. Together, they form the
complete list of VCP because the normed algebra structures can be recovered
from the VCP structures. In particular,
\[
Aut\left(  \mathbb{R}^{7},\times\right)  =Aut_{alg}\left(  \mathbb{O}\right)
\text{.}%
\]
Hence a $G_{2}$ manifold is simply a $7$-dimensional Riemannian manifold
$\left(  M,g\right)  $ with holonomy group inside the exceptional Lie group
$G_{2}=Aut_{alg}\left(  \mathbb{O}\right)  $.

Equivalently, $G_{2}$ manifolds are $7$-dimensional Riemannian manifolds
$\left(  M,g\right)  $ with a \emph{nondegenerate }$3$\emph{-form }$\Omega$
such that $\nabla\Omega=0$. The relation between $\Omega$ and $\times$ is
\[
\Omega\left(  u,v,w\right)  =g\left(  u\times v,w\right)  \text{,
}\ \ \text{for }u,v,w\in T_{p}M.
\]

\begin{example}
\label{R7}(Linear case) The $G_{2}$ manifold $\operatorname{Im}\mathbb{O\simeq
R}^{7}$: Let
\[
\mathbb{R}^{7}\simeq\operatorname{Im}\mathbb{O}\simeq\operatorname{Im}%
\mathbb{H\oplus H=}\left\{  \left(  x_{1}\mathbf{i}+x_{2}\mathbf{j}%
+x_{3}\mathbf{k},x_{4}+x_{5}\mathbf{i}+x_{6}\mathbf{j}+x_{7}\mathbf{k}\right)
\right\}  ,
\]
and the standard basis $e_{i}=\frac{\partial}{\partial x_{i}}\left(
i=1,2,\cdots7\right)  $. The vector cross product $\times$ for $u,v\in
\operatorname{Im}\mathbb{O}$ is defined by $\left(  \ref{cross-product}%
\right)  $.The standard $G_{2}$ $3$-form $\Omega=\Omega_{0}$ \ is%
\[
\Omega_{0}=\omega^{123}-\omega^{167}-\omega^{527}-\omega^{563}-\omega
^{154}-\omega^{264}-\omega^{374},
\]
where $\omega^{ijk}=dx^{i}\wedge dx^{j}\wedge dx^{k}$.
\end{example}

\begin{example}
\label{Calabi-Yau-G2}(Product case) $M=X\times S^{1}$ is a $G_{2}$ manifold if
and only if $X$ is a Calabi-Yau threefold. The $G_{2}$ $3$-form $\Omega$ of
$X$ is related to the holomorphic volume form $\Omega_{X}$ and the K\"{a}hler
form $\omega_{X}$ of $X$ as follow
\[
\Omega=\operatorname{Re}\Omega_{X}+\omega_{X}\wedge d\theta,
\]
where $d\theta$ is the standard angular-form on $S^{1}$.
\end{example}

\bigskip

So far all compact irreducible $G_{2}$ manifolds are constructed by solving
nonlinear PDE of the $G_{2}$ metric using the implicit function theorem,
including (i) resolving orbifold singularities by Joyce \cite{Joyce G2},
\cite{Joyce SH} and (ii) twisted connected sum by Kovalev \cite{CHNP} and
Corti-Haskins-Nordstrom-Pancini \cite{Ko}.

\bigskip

We remark a useful construction of local $G_{2}$ frames.

\begin{remark}
(Cayley-Dickson construction) \label{Cayley-Dickson} A convenient basis in
$\operatorname{Im}\mathbb{O}$ can be constructed inductively (Lemma A.15 in
\cite{Harvey Lawson}): Given orthogonal unit vectors $v_{1}$ and $v_{2}$, we
let $v_{3}:=v_{1}\times v_{2}$; Then we take any unit vector $v_{4}\perp
$span$\left\{  v_{1},v_{2},v_{3}\right\}  $, and let $v_{i+4}:=v_{i}\times
v_{4}$ for $i=1,2,3$. Then $\left\{  v_{i}\right\}  _{i=1}^{7}$ is an
orthonormal basis of $\operatorname{Im}\mathbb{O}$ satisfying the same
$\times$ relation as the standard basis $\left\{  e_{i}\right\}  _{i=1}^{7}$
in $\operatorname{Im}\mathbb{O}$.

This basis $\left\{  v_{i}\right\}  _{i=1}^{7}$ is useful in many local
calculations in \cite{LWZ1} (see also \cite{Akbulut Salur}), where $\left\{
v_{1},v_{2}\right\}  $ span the tangent spaces of a $J$-holomorphic curve.
\end{remark}

\subsection{Instantons (associative submanifolds)\label{instanton}}

A $3$-dimensional submanifold $A$ in a $G_{2}$ manifold $M$ is called an
\textbf{associative submanifold} (or an \textbf{instanton}) if its tangent
space $TA$ is close under $\times$. This notion was introduced by
Harvey-Lawson (\cite{Harvey Lawson}, see also \cite{Lee Leung Instanton
Brane}). The following are two other equivalent conditions for $A$ to be an
associative submanifold: (1) $A$ is calibrated by $\Omega$, namely
$\Omega|_{A}=dv_{A}$; (2) $\tau|_{A}=0\in\Omega^{3}\left(  A,T_{M}%
|_{A}\right)  $ where $\tau\in\Omega^{3}\left(  M,TM\right)  $ is defined by
the following equation%
\begin{equation}
g\left(  \tau\left(  u,v,w\right)  ,z\right)  =\left(  \ast\Omega\right)
\left(  u,w,w,z\right)  \, \label{tau}%
\end{equation}
for any $u,v,w$ and $z\in T_{p}M$, where $\ast\Omega\in\Omega^{4}\left(
M\right)  $ is the Hodge-$\ast$ of $\Omega$. This measurement $\tau|_{A}$ of
associativity is important for perturbing almost instantons to an instanton.

\begin{example}
\label{graph-instanton}\label{ImH}In $\operatorname{Im}\mathbb{O\simeq
}\operatorname{Im}\mathbb{H\oplus H}$, the subspace $\operatorname{Im}%
\mathbb{H\oplus}\left\{  0\right\}  \mathbb{\ }$is an instanton. Explicitly
$\tau$ is ($\left(  5.4\right)  $ in \cite{McLean})%
\begin{align}
\tau &  =\left(  \omega^{256}-\omega^{247}+\omega^{346}-\omega^{357}\right)
\partial_{1}+\left(  \omega^{156}-\omega^{147}-\omega^{345}+\omega
^{367}\right)  \partial_{2}\nonumber\\
&  +\left(  \omega^{245}-\omega^{267}-\omega^{146}-\omega^{157}\right)
\partial_{3}+\left(  \omega^{567}-\omega^{127}+\omega^{136}-\omega
^{235}\right)  \partial_{4}\nonumber\\
&  +\left(  \omega^{126}-\omega^{467}+\omega^{137}+\omega^{234}\right)
\partial_{5}+\left(  \omega^{457}-\omega^{125}-\omega^{134}+\omega
^{237}\right)  \partial_{6}\nonumber\\
&  +\left(  \omega^{124}-\omega^{456}-\omega^{135}-\omega^{236}\right)
\partial_{7}, \label{tau-standard}%
\end{align}
where $\partial_{i}=\frac{\partial}{\partial x^{i}}$. For a map $V=\left(
V^{4},V^{5},V^{6},V^{7}\right)  :\operatorname{Im}\mathbb{H\rightarrow H}$ ,
we let
\[
A=\text{graph}\left(  V\right)  =\left\{  \left(  x_{1},x_{2},x_{3},V\left(
x_{1},x_{2},x_{3}\right)  \right)  |\left(  x_{1},x_{2},x_{3}\right)
\in\operatorname{Im}\mathbb{H}\right\}  .
\]
By $\left(  \ref{tau-standard}\right)  $, the condition $\tau|_{A}=0$ becomes
$V^{\ast}\tau=0$ in coordinates $\left(  x_{1},x_{2},x_{3}\right)  $, and we
see the instanton equation for $A$ is a\emph{\ first order cubic PDE} system
involving product terms $\frac{\partial V^{i}}{\partial x^{1}}\frac{\partial
V^{j}}{\partial x^{2}}\frac{\partial V^{k}}{\partial x^{3}}$.
\end{example}

\begin{example}
For any holomorphic curve $\Sigma$ in a Calabi-Yau manifold $X$, $\Sigma\times
S^{1}$ is an instanton in the $G_{2}$ manifold $X\times S^{1}$.
\end{example}

The \emph{deformations} of an instanton $A$ in a $G_{2}$ manifold $M$ are
governed by a\emph{\ twisted Dirac operator} on its normal bundle $N_{A/M}$,
regarded as a Clifford module over $TA$ with $\times$ as the Clifford multiplication.

\begin{theorem}
\label{deform-instanton}(\cite{McLean}) Infinitesimal deformations of
instantons at $A$ are parameterized by the space of harmonic twisted spinors on
$A$, i.e. kernel of the twisted Dirac operator.
\end{theorem}

\begin{example}
For the instanton $\operatorname{Im}\mathbb{H\oplus}\left\{  0\right\}
\subset\operatorname{Im}\mathbb{O\,\ }$in Example \ref{ImH}, regarding
sections $V=V^{4}+\mathbf{i}V^{5}+\mathbf{j}V^{6}+\mathbf{k}V^{7}$ in its
normal bundle as $\mathbb{H}$-valued functions, then the twisted Dirac
operator is $\mathcal{D}=\mathbf{i}\nabla_{1}+\mathbf{j}\nabla_{2}%
+\mathbf{k}\nabla_{3}$, and%
\begin{align}
\mathcal{D}V  &  =-\left(  V_{1}^{5}+V_{2}^{6}+V_{3}^{7}\right)
+\mathbf{i}\left(  V_{1}^{4}+V_{3}^{6}-V_{2}^{7}\right) \nonumber\\
&  +\mathbf{j}\left(  V_{2}^{4}-V_{3}^{5}+V_{1}^{7}\right)  +\mathbf{k}\left(
V_{3}^{4}+V_{2}^{5}-V_{1}^{6}\right)  . \label{t-Dirac}%
\end{align}

\end{example}

\begin{remark}
Constructing instantons in $G_{2}$ manifolds in general is difficult, partly
because the deformation theory can be obstructed. Our main theorem in
\cite{LWZ1} provides a construction of instantons from $J$-holomorphic curves
in coassociative submanifolds. Other constructions are in \cite{CHNP}.
\end{remark}

\subsection{Boundary value problem on coassociative
submanifolds\label{bdry-value-problem}}

Coassociative submanifolds are analogues of Lagrangian submanifolds in
symplectic geometry. \newline

\begin{definition}
A $4$-dimensional submanifold $C$ in a $G_{2}$ manifold $\left(
M,g,\Omega\right)  $ is called coassociative if $\Omega|_{C}=0$. Equivalently,
$C$ is calibrated by $\ast\Omega$.
\end{definition}

\begin{example}
When $M=X\times S^{1}$ is a product with $X$ a Calabi-Yau threefold, then (i)
$C=L\times S^{1}$ is coassociative if and only if $L\subset X$ is a special
Lagrangian submanifold $L$ of phase $\pi/2$ and (ii) $C=S\times\left\{
\theta\right\}  $ is coassociative if and only if $S\subset X$ is a complex surface.
\end{example}

Deformations of a coassociative submanifold $C$ is studied by McLean
\cite{McLean}: Infinitesimally, they are parameterized by $H_{+}^{2}\left(
C\right)  $, the space of self-dual harmonic $2$-forms on $C$, and they are
always unobstructed.

Finding instantons with boundaries on coassociative submanifolds is an
elliptic problem, similar to finding $J$-holomorphic curves with boundaries on
Lagrangian submanifolds in symplectic geometry.

\begin{theorem}
(Theorem 4.2, \cite{GayetWitt})The linearization of the instanton equation on
an instanton $A$ with boundaries $\Sigma$ lying on a coassociative submanifold
$C$ is an elliptic Fredholm operator, with the Fredholm index given by%
\[
\int_{\Sigma}c_{1}\left(  N_{\Sigma/C}\right)  +1-g\left(  \Sigma\right)  .
\]

\end{theorem}

We have the following orthogonal decomposition of $TM$ along $\Sigma$,%

\[
TM|_{\Sigma}=\mathbb{R}\left\langle n\right\rangle \oplus T\Sigma\oplus
N_{\Sigma/C}\oplus N_{\left\langle C,n\right\rangle /M}|_{\Sigma}\text{,}%
\]
where $n$ be the unit inward normal vector field of $\Sigma$ in $A$. We call
$N_{\Sigma/C}$ the \emph{\textquotedblleft intrinsic\textquotedblright} normal
bundle of $\Sigma$ in $C$ and $N_{\left\langle C,n\right\rangle /M}$ the
\emph{\textquotedblleft extrinsic\textquotedblright} normal bundle defined as
the orthogonal complement of $TC\oplus\left\langle n\right\rangle $ in $TM$.
In \cite{GayetWitt}, Gayet-Witt showed that (1) $T\Sigma$, $N_{\Sigma/C}$ and
$N_{\left\langle C,n\right\rangle /M}|_{\Sigma}$ are complex line bundles with
respect to the almost complex structure $J_{n}:=$ $n\times$ and (2)
$N_{\Sigma/C}\otimes_{\mathbb{C}}\wedge_{\mathbb{C}}^{0,1}\left(  T^{\ast
}\Sigma\right)  \simeq N_{\left\langle C,n\right\rangle /M}^{\ast}$ as complex
line bundles over $\Sigma$ which is obtained by changing the tensor product
$\otimes$ to the vector cross product $\times$, and used metric to identify
conjugate bundle with dual bundle. This relationship between intrinsic and
extrinsic geometry of $\Sigma$ in $C$ is needed (\cite{LWZ1}) in order to show
two natural Dirac operators on $\Sigma$ agree which is needed in section
\ref{2Dirac} for the proof of our main theorem.

\bigskip

The following table gives an interesting comparison between $G_{2}$-geometry
and symplectic geometry.%

\[%
\begin{tabular}
[c]{|l|l|}\hline
$G_{2}$\textbf{\ manifold }$%
\begin{array}
[c]{c}%
\\
\end{array}
$ & \textbf{Symplectic manifold }\\\hline
nondegenerate $3$-form $\Omega$, $d\Omega=0%
\begin{array}
[c]{c}%
\\
\end{array}
$ & nondegenerate $2$-form $\omega$, $d\omega=0$\\\hline
vector cross product $\times%
\begin{array}
[c]{c}%
\\
\end{array}
$ & almost complex structure $J$\\\hline
$\Omega\left(  u,v,w\right)  =g\left(  u\times v,w\right)
\begin{array}
[c]{c}%
\\
\end{array}
$ & $\omega\left(  u,v\right)  =g\left(  Ju,v\right)  $\\\hline
$%
\begin{array}
[c]{c}%
\text{instanton }A\\
TA\text{ preserved by}\times\\
A\text{ calibrated by }\Omega
\end{array}
$ & $%
\begin{array}
[c]{c}%
J\text{-holomorphic curve }\Sigma\\
T\Sigma\text{ preserved by }J\\
\Sigma\text{ calibrated by }\omega
\end{array}
\ $\\\hline
$%
\begin{array}
[c]{c}%
\text{coassociative submanifold }C\\
\Omega|_{C}=0\text{ and }\dim C=4
\end{array}
$ & $%
\begin{array}
[c]{c}%
\text{Lagrangian submanifold }L\\
\omega|_{L}=0\text{ and }\dim L=\frac{1}{2}\dim X
\end{array}
$\\\hline
\end{tabular}
\]

\subsection{Intersection theory?\label{intersect}}

Intersection theory of Lagrangian submanifolds is an essential part of
symplectic geometry. By counting the number of holomorphic disks bounding
intersecting Lagrangian submanifolds, Floer and others defined the celebrated
Floer homology theory. It plays an important role in mirror symmetry for
Calabi-Yau manifolds and string theory in physics.

In M-theory, Calabi-Yau threefolds are replaced by $G_{2}$-manifolds. The
analogs of holomorphic disks (resp. Lagrangian submanifolds) are instantons or
associative submanifolds (resp. coassociative submanifolds or branes) in $M$
\cite{Lee Leung Instanton Brane}. The problem of counting instantons has been
discussed by many physicists. For example Harvey and Moore \cite{Harvey Moore}
discussed the mirror symmetry aspects of it; Aganagic and Vafa \cite{Mina
Vafa} related it to the open Gromov-Witten invariants for local Calabi-Yau
threefolds; Beasley and Witten \cite{BeasleyWitten} argued that one should
count instantons using the Euler characteristic of their moduli spaces.

On the mathematical side, the compactness issues of the moduli of instantons
is a very challenging problem because the bubbling-off phenomena of
($3$-dimensional) instantons is not well understood. This makes it very
difficult to define an honest invariant by counting instantons. Note the
Fredholm theory for instantons with coassociative boundary conditions has been
set up \cite{GayetWitt}.

In symplectic geometry, Fukaya and Oh \cite{Fukaya Oh} considered two
\emph{nearby} Lagrangian submanifolds $L$ and $L_{t}$, where $L_{t}$ is the
graph of a closed $1$-form $\alpha$ on $L$ scaled by a small $t$, then
holomorphic disks bounding them with \emph{small volume} is closely related to
gradient flow lines of $\alpha$. They actually deal with $J$-holomorphic
polygons bounding $k$ Lagrangians and need to smooth vertex singularities on
gradient flow trees when $k\geq3$. For simplicity we only state their result
for the $k=2$ case here.

Given a closed $1$-form $\alpha$ on $L$, $L_{t}:=$graph$\left(  t\alpha
\right)  $ is a Lagrangian submanifold in $T^{\ast}L$ near the zero section
$L$. By Weinstein's neighborhood theorem, a small tubular neighborhood of a
Lagrangian $L$ in a symplectic manifold can be identified as a tubular
neighborhood of $L$ in $T^{\ast}L$.

\begin{theorem}
(\cite{Fukaya Oh}) For any compact Riemannian manifold $\left(  L,g\right)  $,
$T^{\ast}L$ has a natural almost K\"{a}hler structure $\left(  \omega
,J_{g}\right)  $.\ Let $L_{t}:=$graph$\left(  t\alpha\right)  $ be a
Lagrangian submanifold in $T^{\ast}L$ for a closed $1$-form $\alpha$ on $L$.
There is a constant $\varepsilon_{0}>0$, such that for any $t\in
(0,\varepsilon_{0}]$, there is a $1$-$1$ correspondence between $J_{g}%
$-holomorphic curves bounding $L\cup L_{t}$ and gradient flow lines of
$\alpha$ on $L$.
\end{theorem}

We want to build the following analogue: instantons $A$ bounding $C\cup C_{t}$
have $1$-$1$ correspondence to $J_{n}$-holomorphic curves $\Sigma$ on $C$.%

\[%
\begin{tabular}
[c]{|l|l|}\hline
\textbf{Symplectic manifold }$M%
\begin{array}
[c]{c}%
\\
\end{array}
$ & $G_{2}$\textbf{\ manifold }$M$\\\hline
Lagrangian submanifolds $L$ and $L_{t}%
\begin{array}
[c]{c}%
\\
\end{array}
$ & coassociative submanifolds $C$ and $C_{t}$\\\hline
$J$-holo. curve bounding $L\cup L_{t}%
\begin{array}
[c]{c}%
\\
\end{array}
$ & instanton $A$ bounding $C\cup C_{t}$\\\hline
gradient flow line of $\alpha$ on $L%
\begin{array}
[c]{c}%
\\
\end{array}
$ & $J_{n}$-holomorphic curve $\Sigma$ on $C$\\\hline
\end{tabular}
\
\]
The meaning of $J_{n},C_{t}$ and the precise statment of our result will be
explained in the next section.

\section{Main Theorem: instantons from ${\protect\Large J}$%
-{\protect\Large holomorphic curves\label{thin-instanton}}}

\subsection{Statement of the main theorem\label{statement-main-theorem}}

Let $\mathcal{C}=\cup_{0\leq t\leq\varepsilon}C_{t}$ be a family of
coassociative manifolds $C_{t}$ in a $G_{2}$-manifold $M$, regarded as a
deformation of $C=C_{0}$ along the normal vector field $n:=\frac{dC_{t}}%
{dt}|_{t=0}$. Then \footnote{Here $\iota_{n}$ is the contraction of a form by
the vector field $n$.}$\iota_{n}\Omega$ is a self-dual harmonic $2$-form on
$C$ by McLean's Theorem on deformations of coassociative submanifolds (Section
\ref{bdry-value-problem}). In particular $\omega_{n}:=\iota_{n}\Omega$ defines
a \emph{symplectic structure }on $C\backslash n^{-1}\left(  0\right)  $ as
$\omega_{n}\wedge\omega_{n}=\ast\omega_{n}\wedge\omega_{n}=\left\vert
\omega_{n}\right\vert ^{2}dv_{C}$ is nonzero outside $\left\{  n=0\right\}  $.
Furthermore $J_{n}:=\frac{n}{\left\vert n\right\vert }\times$ defines a
compatible \emph{almost complex structure} on $\left(  C\backslash
n^{-1}\left(  0\right)  ,\omega_{n}\right)  $. When $n$ has no zero, we have
the following main theorem in \cite{LWZ1}

\begin{theorem}
\label{correspondence}Suppose that $\left(  M,\Omega\right)  $\ is a $G_{2}%
$-manifold and $\left\{  C_{t}\right\}  $\ is an one-parameter smooth family
of coassociative submanifolds in $M$. When $\iota_{n}\Omega\in\Omega_{+}%
^{2}\left(  C_{0}\right)  $\ is nonvanishing, then

\begin{enumerate}
\item (Proposition 6) If $\left\{  \mathtt{A}_{t}\right\}  $\ is any
one-parameter family of associative submanifolds (i.e. instantons) in
$M$\ satisfying
\[
\partial\mathtt{A}_{t}\subset C_{t}\cup C_{0},\text{ }\lim_{t\rightarrow
0}\mathtt{A}_{t}\cap C_{0}=\Sigma_{0}\text{ in the }C^{1}\text{-topology,}%
\]
then $\Sigma_{0}$\ is a $J_{n}$-holomorphic curve in $C_{0}$. \

\item (Theorem 24) Conversely, every regular $J_{n}$-holomorphic curve
$\Sigma_{0}$ (namely those for which the linearization of $\overline{\partial
}_{J_{n}}$ on $\Sigma_{0}$ is surjective) in $C_{0}\ $is the limit of a family
of associative submanifolds $\mathtt{A}_{t}$'s as described above.
\end{enumerate}
\end{theorem}

Our results are similar to those in Fukaya-Oh \cite{Fukaya Oh} and the proofs
also share some similarities: relating the Fredholm regular property of higher
dimensional linearized instanton equations to lower dimensional ones;
necessity to deal with the lack of uniform ellipticity as the domain collapses
when $\varepsilon\rightarrow0$; using the periodic reflection technique to
\textquotedblleft thicken\textquotedblright\ the collapsing domain to achieve
a uniform right inverse estimate in the $W^{1,p}$ setting.

However the proof in our case has more difficulties than those needed in the
$k=2$ case of \cite{Fukaya Oh}, as explained in the introduction. For $k\geq
3$, \cite{Fukaya Oh} contains difficulties we have not encountered here: to
find the local models of the singularities of degenerating $J$-holomorphic
polygons and resolve them.

\begin{remark}
Given any Riemann surface $\Sigma\subset M,$ it can always be thickened to a
instanton by the Cartan-K\"{a}hler theory (\cite{Harvey Lawson}%
,\cite{GayetWitt}). However its boundary may not lie inside any coassociative
submanifold (see \cite{GayetWitt}). In our case, we produce an instanton
$A_{\varepsilon}$ with boundary in the coassociative submanifold $C\cup
C_{\varepsilon}$, but $\partial A_{\varepsilon}\cap C$ is only close but not
equal to $\Sigma$.
\end{remark}

\subsection{Main ingredients of the proof\label{proof}}

\subsubsection{Formulating the instanton equation near an almost
instanton\label{nonlinear_instanton_eq}}

We first produce an almost instanton with boundaries on $C_{0}\cup
C_{\varepsilon}$. Let $\varphi:\left[  0,\varepsilon\right]  \times
C\rightarrow M$ be a parametrization of the family of coassociative
submanifolds $\left\{  C_{t}\right\}  _{0\leq t\leq\varepsilon}$. Under the
assumptions that $n=\frac{dC_{t}}{dt}|_{t=0}$ is nonvanishing and $\Sigma$ is
a $J_{n}$-holomorphic curve in $C=C_{0}$, we define
\[
A_{\varepsilon}:=\left[  0,\varepsilon\right]  \times\Sigma\text{, and
}A_{\varepsilon}^{\prime}=\varphi\left(  A_{\varepsilon}\right)  \subset M,
\]
then $A_{\varepsilon}^{\prime}$ is an \emph{almost instanton} with $\partial
A_{\varepsilon}^{\prime}\subset C_{0}\cup C_{\varepsilon}$ in the sense that
$\left\vert \tau|_{A_{\varepsilon}^{\prime}}\right\vert _{C^{0}}$ is small.
Recall that a $3$-dimensional submanifold $A\subset M$ is an instanton if and
only if $\tau|_{A}=0$, where $\tau\in\Omega^{3}\left(  M,TM\right)  $ is
defined in $\left(  \ref{tau}\right)  $. The reason for the smallness is that,
at $p\in$ $\varphi\left(  \left\{  0\right\}  \times\Sigma\right)  \subset$
$A_{\varepsilon}^{\prime}$, $T_{p}A_{\varepsilon}^{\prime}$ is associative by
the $J_{n}$-holomorphic property of $\Sigma$ so $\tau|_{T_{p}A_{\varepsilon
}^{\prime}}=0$, and any point $q$ on $A_{\varepsilon}^{\prime}$ has
$\varepsilon$-order distance to $\varphi\left(  \left\{  0\right\}
\times\Sigma\right)  $ while $\tau|_{T_{q}A_{\varepsilon}^{\prime}}$ smoothly
depends on $q$.

Next we formulate $\tau|_{A}$ as a nonlinear map $F_{\varepsilon}$ on the
space $\Gamma\left(  A_{\varepsilon}^{\prime},N_{A_{\varepsilon}^{\prime}%
/M}\right)  $ of sections of the \emph{normal bundle }$N_{A_{\varepsilon
}^{\prime}/M}$\emph{\ of }$A_{\varepsilon}^{\prime}$, in particular
$F_{\varepsilon}\left(  V\right)  =0$ if and only if%
\begin{equation}
\mathtt{A}_{\varepsilon}\left(  V\right)  :=\left(  \exp V\right)  \left(
\mathtt{A}_{\varepsilon}^{\prime}\right)
\end{equation}
is an instanton where $\exp V:\mathtt{A}_{\varepsilon}^{\prime}\rightarrow M$.

To do this, we let%

\begin{align*}
C^{\alpha}\left(  \mathtt{A}_{\varepsilon}^{\prime},N_{\mathtt{A}%
_{\varepsilon}^{\prime}/M}\right)   &  =\left\{  V\in\Gamma\left(
\mathtt{A}_{\varepsilon}^{\prime},N_{\mathtt{A}_{\varepsilon}^{\prime}%
/M}\right)  |V\text{ }\in C^{\alpha}\right\}  ,\\
C_{-}^{1,\alpha}\left(  \mathtt{A}_{\varepsilon}^{\prime},N_{\mathtt{A}%
_{\varepsilon}^{\prime}/M}\right)   &  =\left\{  V\in\Gamma\left(
\mathtt{A}_{\varepsilon}^{\prime},N_{\mathtt{A}_{\varepsilon}^{\prime}%
/M}\right)  |V\text{ }\in C^{1,\alpha}\text{, }V|_{\partial\mathtt{A}%
_{\varepsilon}^{\prime}}\subset TC_{0}\cup TC_{\varepsilon}\right\}  ,
\end{align*}
(the \textquotedblleft$-$\textquotedblright\ in $C_{-}^{1,\alpha}$ is for the
coassociative boundary condition), and%
\[
F_{\varepsilon}:C_{-}^{1,\alpha}\left(  \mathtt{A}_{\varepsilon}^{\prime
},N_{\mathtt{A}_{\varepsilon}^{\prime}/M}\right)  \rightarrow C^{\alpha
}\left(  \mathtt{A}_{\varepsilon}^{\prime},N_{\mathtt{A}_{\varepsilon}%
^{\prime}/M}\right)  ,
\]
\[
F_{\varepsilon}\left(  V\right)  =\ast_{\mathtt{A}_{\varepsilon}^{\prime}%
}\circ\bot_{\mathtt{A}_{\varepsilon}^{\prime}}\circ\left(  T_{V}\circ\left(
\exp V\right)  ^{\ast}\tau\right)  ,
\]
where

\begin{enumerate}
\item $\left(  \exp V\right)  ^{\ast}$ pulls back the \emph{differential form
part} of the tensor $\tau$

\item $T_{V}:T_{\exp_{p}\left(  tV\right)  }M\rightarrow T_{p}M$ pulls back
the \emph{vector part} of $\tau$ by parallel transport along the geodesic
$\exp_{p}\left(  tV\right)  $

\item $\bot_{\mathtt{A}_{\varepsilon}^{\prime}}:TM|_{\mathtt{A}_{\varepsilon
}^{\prime}}\rightarrow N_{\mathtt{A}_{\varepsilon}^{\prime}/M}$ $\ $is the
orthogonal$\ $projection with respect to $g$

\item $\ast_{\mathtt{A}_{\varepsilon}^{\prime}}:\Omega^{3}\left(
\mathtt{A}_{\varepsilon}^{\prime}\right)  \rightarrow\Omega^{0}\left(
\mathtt{A}_{\varepsilon}^{\prime}\right)  $ \ is the$\ $quotient by the volume
form $dvol_{\mathtt{A}_{\varepsilon}^{\prime}}$
\end{enumerate}

If $\mathtt{A}_{\varepsilon}^{\prime}$ is an almost instanton, then a $G_{2}%
$-linear algebra argument shows that when $\left\Vert V\right\Vert
_{C^{1,\alpha}\left(  A_{\varepsilon}^{\prime},N_{A_{\varepsilon}^{\prime}%
/M}\right)  }\ $is small we have
\[
F_{\varepsilon}\left(  V\right)  =0\text{ }\Longleftrightarrow\mathtt{A}%
_{\varepsilon}\left(  V\right)  \text{ is an instanton.}%
\]
To ensure that $\mathtt{A}_{\varepsilon}\left(  V\right)  $ satisfies the boundary
condition, in the definition of $\exp V$ we actually need to modify the metric
$g$ near $C_{0}\cup C_{\varepsilon}$ to make them totally geodesic, but we
will keep the original metric in $T_{V}$, $\bot_{\mathtt{A}_{\varepsilon
}^{\prime}}$ and $\ast_{\mathtt{A}_{\varepsilon}^{\prime}}$. This modification
will not change the expression of $F_{\varepsilon}^{\prime}\left(  0\right)  $
$\left(  \ref{eq:lin-instanton}\right)  $ (see Remark 10 (1) in \cite{LWZ1}).
So our estimate for $\left\Vert F_{\varepsilon}^{\prime}\left(  0\right)
^{-1}\right\Vert $ is still valid in the new metric.

\subsubsection{Linearizing the instanton equation using a good
frame\label{subsec:linearization}}

We can make $F_{\varepsilon}\left(  V\right)  $ more explicit by using a good
local frame field $\left\{  W_{\alpha}\right\}  _{\alpha=1}^{7}$ near $p\in
A_{\varepsilon}^{\prime}=\mathtt{A}_{\varepsilon}\left(  0\right)  $. Since
$\tau$ is a vector-valued $3$-form, following Einstein's summation convention,
we can write
\[
\tau=\omega^{\alpha}\otimes W_{\alpha}\in\Omega^{3}\left(  M,TM\right)
\text{,}%
\]
where local 3-forms $\left\{  \omega^{\alpha}\right\}  _{\alpha=1}^{7}$ are
determined by $\left\{  W_{\alpha}\right\}  _{\alpha=1}^{7}$. For any
$q=\exp_{p}V\in\mathtt{A}_{\varepsilon}\left(  V\right)  $, we let
\[
F\left(  V\right)  \left(  p\right)  =\left(  \exp V\right)  ^{\ast}%
\omega^{\alpha}\left(  p\right)  \otimes T_{V}W_{\alpha}\left(  q\right)  .
\]
$F\left(  V\right)  $ is the essential part of $F_{\varepsilon}\left(
V\right)  $ as $\ast_{\mathtt{A}_{\varepsilon}^{\prime}}\circ\bot
_{\mathtt{A}_{\varepsilon}^{\prime}}$ is only a projection. It is
\emph{independent }on the choice of frame $\left\{  W_{\alpha}\right\}
_{\alpha=1}^{7}$. To compute $F^{\prime}\left(  0\right)  $, we have the
following \emph{generalized Cartan formula} $\left(  \ref{vec-form-derivative}%
\right)  $: \
\begin{align}
&  F^{\prime}\left(  0\right)  V\left(  p\right) \nonumber\\
&  =L_{V}\omega^{\alpha}\otimes W_{\alpha}+\omega^{\alpha}\otimes\nabla
_{V}W_{\alpha}\nonumber\\
&  =\underset{\text{symbol part}}{\underbrace{d\left(  i_{V}\omega^{\alpha
}\right)  \otimes W_{\alpha}}}+\underset{0\text{-th order part}}%
{\underbrace{i_{V}d\omega^{\alpha}\left(  p\right)  \otimes W_{\alpha}%
+\omega^{\alpha}\otimes\nabla_{V}W_{\alpha}\left(  p\right)  }}%
.\label{vec-form-derivative}\\
&  =d\left(  i_{V}\omega^{\alpha}\right)  \otimes W_{\alpha}\left(  p\right)
+B_{\alpha}\left(  p\right)  V\otimes W_{\alpha}+\omega^{\alpha}\otimes
C_{\alpha}\left(  p\right)  V\nonumber\\
&  =\mathcal{D}V\otimes vol_{A_{\varepsilon}^{\prime}}+E\left(  p\right)
\left(  V\right)  \label{eq:lin-instanton}%
\end{align}
where $B_{\alpha}$, $C_{\alpha}$ and $E$ are certain matrix-valued functions.

\bigskip We require a \emph{\textquotedblleft good\textquotedblright\ }frame
field $\left\{  W_{\alpha}\right\}  _{\alpha=1}^{7}$ to satisfy the following
conditions in any small $\varepsilon$\emph{-ball }around\emph{\ }$p$\emph{:}

\begin{enumerate}
\item $B_{\alpha}$, $C_{\alpha}$ and $E$ are of $\varepsilon$-order in $C^{1}$ norm,

\item $\omega^{a}$'s are $\varepsilon$-close to $\omega^{ijk}$'s in $\left(
\ref{tau-standard}\right)  $ for $\mathbb{R}^{7}$ in $C^{1}$ norm, in the
sense that $\omega^{ijk}$ are replaced by $W_{i}^{\ast}\wedge W_{j}^{\ast
}\wedge W_{k}^{\ast}$, where $W_{i}^{\ast}$ is the dual vector of $W_{i}$,

\item $\mathcal{D}V\left(  p\right)  $ is $\varepsilon$-close to the twisted
Dirac operator $\left(  \ref{t-Dirac}\right)  $ for $\mathbb{R}^{7}$ in
$C^{1}$ norm, in the sense that $\frac{\partial}{\partial x^{i}}$ in $\left(
\ref{t-Dirac}\right)  $ are replaced by $\nabla_{W_{i}}^{\perp}$.
\end{enumerate}

Condition $1$ holds when $\left\{  W_{\alpha}\right\}  _{\alpha=1}^{7}$ is
parallel along the normal bundle directions. Conditions $2$ holds if whenever
$e_{\gamma}=e_{\alpha}\times e_{\beta}$ in $\operatorname{Im}\mathbb{O}$ we
have $\left\Vert W_{\gamma}-W_{\alpha}\times W_{\beta}\right\Vert _{C^{1}%
}=O\left(  \varepsilon\right)  $. Condition $3$ holds if we further have the
\emph{normal covariant derivatives} $\left\Vert \nabla_{W_{i}}^{\perp}%
W_{k}\right\Vert _{C^{1}}=O\left(  \varepsilon\right)  $ assuming that $W_{i}%
$'s span $TA_{\varepsilon}^{\prime}$ and $W_{k}$'s span $N_{A_{\varepsilon
}^{\prime}/M}$.

\bigskip Such a \textit{good }frame $\left\{  W_{\alpha}\right\}  _{\alpha
=1}^{7}$ can be constructed by the Cayley-Dickson construction as explained in
Remark \ref{Cayley-Dickson}. The principal part of the linearized instanton
equation $F^{\prime}\left(  0\right)  V$ on $A_{\varepsilon}^{\prime}$ is the
term $\mathcal{D}V$ in $\left(  \ref{eq:lin-instanton}\right)  $, which is a
first order differential operator with a nice geometric meaning (see the next subsection).

\subsubsection{A simplified model: Dirac operators on thin manifolds}

\bigskip We temporarily leave $G_{2}$ geometry and consider a
(\emph{\textquotedblleft thin\textquotedblright\ }when $\varepsilon$ is small)
$3$-manifold
\[
\mathtt{A}_{\varepsilon}:=\left[  0,\varepsilon\right]  \times\Sigma=\left\{
\left(  x_{1},z:=x_{2}+ix_{3}\right)  \right\}  ,
\]
with the \emph{warped product metric}
\[
g_{\mathtt{A}_{\varepsilon},h}=h\left(  z\right)  dx_{1}^{2}+g_{\Sigma},
\]
where $h\left(  z\right)  =\left\vert n\right\vert ^{2}>0$ and $n=\frac
{dC_{t}}{dt}|_{t=0}$ is the nonvanishing normal vector field on $C_{0}$. This
is a first order approximation of the induced metric on $\mathtt{A}%
_{\varepsilon}^{\prime}\subset M$.

We first consider the geometry of a $J_{n}$-holomorphic curve $\Sigma$ in $C$.
Let $L=N_{\Sigma/C}$ be the normal bundle of $\Sigma$ in $C$, then $L$ is a
Hermitian $J_{n}$-complex line bundle over $\Sigma$ (see Proposition
\ref{Dirac-bdl}). Let $\bar{\partial}=\left(  \overline{\partial}%
,\overline{\partial}^{\ast}\right)  $ be the Dirac operator on the Dolbeault
complex $\Omega_{\mathbb{C}}^{0}\left(  L\right)  \oplus\Omega_{\mathbb{C}%
}^{0,1}\left(  L\right)  $ of the spinor bundle of $\Sigma$%
\[
\mathbb{S}_{\Sigma}=\mathbb{S}_{\Sigma}^{+}\oplus\mathbb{S}_{\Sigma}%
^{-}=L\oplus\wedge_{\mathbb{C}}^{0,1}\left(  L\right)
\]
such that%
\begin{equation}
\Omega_{\mathbb{C}}^{0}\left(  L\right)  \oplus\Omega_{\mathbb{C}}%
^{0,1}\left(  L\right)  \overset{\left(  \overline{\partial},\overline
{\partial}^{\ast}\right)  }{\longrightarrow}\Omega_{\mathbb{C}}^{0,1}\left(
L\right)  \oplus\Omega_{\mathbb{C}}^{0}\left(  L\right)  ,
\label{Dolbeault-Dirac}%
\end{equation}
where $\overline{\partial}:$ $\Omega_{\mathbb{C}}^{0}\left(  L\right)
\rightarrow\Omega_{\mathbb{C}}^{0,1}\left(  L\right)  $ is the \emph{normal
Cauchy-Riemann operator} of $J_{n}$-holomorphic curve $\Sigma$ in $C$, and
$\overline{\partial}^{\ast}:\Omega_{\mathbb{C}}^{0,1}\left(  L\right)
\rightarrow\Omega_{\mathbb{C}}^{0}\left(  L\right)  $ is its adjoint. (In
\cite{LWZ1}, we use the notation $\partial^{+}=-$ $\mathbf{i}\overline
{\partial}^{\ast}$ and $\partial^{-}=\mathbf{i}\overline{\partial}$).

\bigskip

To describe the spinor bundle $\mathbb{S}$ over the $3$-manifold
$\mathtt{A}_{\varepsilon}$, we pullback $\mathbb{S}_{\Sigma}=\mathbb{S}%
_{\Sigma}^{+}\oplus\mathbb{S}_{\Sigma}^{-}$ to $\mathtt{A}_{\varepsilon
}=\left[  0,\varepsilon\right]  \times\Sigma$, and denote it as $\mathbb{S}%
=\mathbb{S}^{+}\oplus\mathbb{S}^{-}$. Let $e_{1}$ be the unit tangent vector
field on $\mathtt{A}_{\varepsilon}$ along $x_{1}$-direction. For $\mathbb{S}$
to be the spinor bundle of $\mathtt{A}_{\varepsilon}$, the Clifford
multiplication with $e_{1}$ should be $\pm i$ on $\mathbb{S}^{\pm}$. To
describe the Dirac operator on $\mathbb{S}$, we define
\[
\mathcal{D}=e_{1}\cdot h^{-\frac{1}{2}}\left(  z\right)  \frac{\partial
}{\partial x_{1}}+\bar{\partial},
\]
where $\bar{\partial}=\left(  \overline{\partial},\overline{\partial}^{\ast
}\right)  $ is the Dirac Dolbeault operator in equation (\ref{Dolbeault-Dirac}).
$\mathcal{D}$ acts on the sections $V=\left(  u,v\right)  $ of $\mathbb{S}%
=\mathbb{S}^{+}\oplus\mathbb{S}^{-}$ over $\mathtt{A}_{\varepsilon}$ with
\emph{local elliptic boundary condition} (in the sense of \cite{BW}):%
\[
v|_{\partial\mathtt{A}_{\varepsilon}}=0.
\]

We can write a local expression of\emph{\ }$\mathcal{D}$. Consider the section
$V=\left(  u,v\right)  $ of $\mathbb{S}$ with $u=V^{4}+\mathbf{i}V^{5}%
\in\mathbb{S}^{+}$ and $v=V^{6}+\mathbf{i}V^{7}\in\mathbb{S}^{-}$ (for
$\mathbb{S}^{\pm}$ are complex line bundles). Then $e_{1}\cdot=\left[
\begin{array}
[c]{cc}%
\mathbf{i} & 0\\
0 & -\mathbf{i}%
\end{array}
\right]  $ and
\begin{align}
\mathcal{D}V  &  =\left[
\begin{array}
[c]{cc}%
\mathbf{i} & 0\\
0 & -\mathbf{i}%
\end{array}
\right]  \left(  h^{-\frac{1}{2}}\left(  z\right)  \frac{\partial}{\partial
x_{1}}+\left[
\begin{array}
[c]{cc}%
0 & \mathbf{i}\partial_{z}\\
\mathbf{i}\bar{\partial}_{z} & 0
\end{array}
\right]  \right)  \left[
\begin{array}
[c]{c}%
u\\
v
\end{array}
\right]  ,\nonumber\\
&  =\left(  h^{-\frac{1}{2}}\left(  z\right)  \frac{\partial u}{\partial
x_{1}}\mathbf{i}-\partial_{z}v\right)  +\left(  -h^{-\frac{1}{2}}\left(
z\right)  \frac{\partial v}{\partial x_{1}}\mathbf{i+}\bar{\partial}%
_{z}u\right)  \cdot\mathbf{j} \label{3mfd-Dirac}%
\end{align}
where $\bar{\partial}_{z}:=\nabla_{2}+\mathbf{i}\nabla_{3}$ and $\partial
_{z}:=\nabla_{2}-\mathbf{i}\nabla_{3}$.

On can check that $\mathcal{D}$ agrees with the linearized instanton equation
on $\left\{  0\right\}  \times\Sigma$, and on A$_{\varepsilon}$ they are very
close (Subsection \ref{exp-like}). This is why we can use $\mathcal{D}$ of the
linear model to study deformations of instantons in $G_{2}$ manifolds. The
precise comparison is in Subsection \ref{2Dirac} and Subsection \ref{exp-like}.

\subsubsection{Key estimates of $\mathcal{D}^{-1}$ of the linear
model\label{D-inverse}}

The most difficult part of \cite{LWZ1} is to derive an explicit $\varepsilon
$-dependent bound of the operator norm of $\mathcal{D}^{-1}:$ $C^{\alpha
}\left(  \mathtt{A}_{\varepsilon},\mathbb{S}\right)  \rightarrow
C_{-}^{1,\alpha}\left(  \mathtt{A}_{\varepsilon},\mathbb{S}\right)  $. This is
nontrivial because we loss uniform ellipticity as $\mathtt{A}_{\varepsilon
}=\left[  0,\varepsilon\right]  \times\Sigma$ collapses to $\Sigma$ and
consequently $\left\Vert \mathcal{D}^{-1}\right\Vert $ may blow up. A good
control of $\left\Vert \mathcal{D}^{-1}\right\Vert $ is crucial for singular
perturbation problems in general. We could achieve this in our case roughly
because of

\begin{enumerate}
\item The Fredholm regularity property of the $J_{n}$-holomorphic curve
$\Sigma$, which supplies the transversality for $\mathcal{D}$,

\item The coassociative boundary condition on $\partial\mathtt{A}%
_{\varepsilon}$, which enables us to periodically reflect $\mathtt{A}%
_{\varepsilon}$ to a bigger domain $\mathtt{A}_{k\left(  \varepsilon\right)
\varepsilon}$ with integer $k\left(  \varepsilon\right)  $ such that $1/2\leq
k\left(  \varepsilon\right)  \varepsilon\leq3/2$, thus restoring the uniform ellipticity.
\end{enumerate}

This is an over-simplified description as sections on $\mathtt{A}%
_{\varepsilon}$ may become \emph{discontinuous} after periodical reflection,
so condition $2$ only helps in the $W^{1,p}$ setting to get a uniform bound of
$\left\Vert \mathcal{D}^{-1}\right\Vert $ as in \cite{Fukaya Oh}. More effort
is needed to estimate $\left\Vert \mathcal{D}^{-1}\right\Vert $ in the
Schauder $C^{1,\alpha}$ setting.

Recall in Subsection \ref{nonlinear_instanton_eq} we have formulated
instantons nearby $A_{\varepsilon}^{\prime}$ as solutions of the nonlinear
equation $F_{\varepsilon}\left(  V\right)  =0$, and it turns out $\left\Vert
F_{\varepsilon}\left(  0\right)  \right\Vert _{C^{\alpha}}\leq C\varepsilon
^{1-\alpha}$. In Subsection \ref{subsec:linearization} we have\ computed the
linearization $F_{\varepsilon}^{\prime}\left(  0\right)  $. To apply the
implicit function theorem to perturb $A_{\varepsilon}^{\prime}$ to a true
instanton, we need to estimate $\left\Vert F_{\varepsilon}^{\prime}\left(
0\right)  ^{-1}\right\Vert $, and we will see in Subsection \ref{exp-like} it
is comparable to $\left\Vert \mathcal{D}^{-1}\right\Vert $ (Proposition
\ref{compare-linear-model}). So a key estimate that we need is the following

\begin{theorem}
\label{e-inverse-bound} ($\varepsilon$-dependent bound) Suppose that the first
eigenvalues for $\overline{\partial}\overline{\partial}^{\ast}$ and
$\overline{\partial}^{\ast}\overline{\partial}$ are bounded below by
$\lambda>0.$ Then for any $0<\alpha<1$ and $p>3$ there is a positive constant
$C=C\left(  \alpha,p,\lambda,h\right)  $ such that for any $V\in
C_{-}^{1,\alpha}\left(  \mathtt{A}_{\varepsilon},\mathbb{S}\right)  $, we
have
\begin{equation}
\left\Vert V\right\Vert _{C_{-}^{1,\alpha}}\leq C\varepsilon^{-\left(
\frac{3}{p}+2\alpha\right)  }\left\Vert \mathcal{D}V\right\Vert _{C^{\alpha}}.
\label{e-right-inverse-estimate}%
\end{equation}

\end{theorem}

Here the notations $C_{\pm}^{1,\alpha}\left(  \mathtt{A}_{\varepsilon
},\mathbb{S}\right)  $ are chosen to indicate the \emph{local elliptic
boundary conditions }for sections\emph{ }$V=\left(  u,v\right)  $ of
$\mathbb{S=S}^{+}\oplus\mathbb{S}^{-}$:
\begin{align*}
C_{-}^{1,\alpha}\left(  \mathtt{A}_{\varepsilon},\mathbb{S}\right)   &
=\left\{  V\in C^{1,\alpha}\left(  \mathtt{A}_{\varepsilon},\mathbb{S}\right)
:v|_{\partial\mathtt{A}_{\varepsilon}}=0\right\} \\
C_{+}^{1,\alpha}\left(  \mathtt{A}_{\varepsilon},\mathbb{S}\right)   &
=\left\{  V\in C^{1,\alpha}\left(  \mathtt{A}_{\varepsilon},\mathbb{S}\right)
:u|_{\partial\mathtt{A}_{\varepsilon}}=0\right\}  .
\end{align*}
Similarly $W_{\pm}^{k,p}\left(  \mathtt{A}_{\varepsilon},\mathbb{S}\right)  $
will be used to indicate these boundary conditions.

To prove this theorem we first recall (for simplicity, we assume $h\equiv1$),
\begin{equation}
\mathcal{D}V=\left[
\begin{array}
[c]{cc}%
\mathbf{i} & 0\\
0 & -\mathbf{i}%
\end{array}
\right]  \left(  \frac{\partial}{\partial x_{1}}+\left[
\begin{array}
[c]{cc}%
0 & -\mathbf{i}\overline{\partial}^{\ast}\\
\mathbf{i}\overline{\partial} & 0
\end{array}
\right]  \right)  \left[
\begin{array}
[c]{c}%
u\\
v
\end{array}
\right]  . \label{3mfd-D}%
\end{equation}
Along $\Sigma$ we have the following

\begin{lemma}
\label{lem:lambda}Assume the $J_{n}$-holomorphic curve $\Sigma\subset C$ is
Fredholm regular, then
\[
\lambda_{\overline{\partial}^{\ast}}>0
\]
where $\lambda_{\overline{\partial}^{\ast}}$ is the first eigenvalue
of\ $\Delta_{\Sigma}=\overline{\partial}\overline{\partial}^{\ast}$ on
$W^{1,2}\left(  \Sigma,\mathbb{S}^{+}\right)  $.
\end{lemma}

\bigskip

\begin{proof}
This follows from the fact that $\overline{\partial}$ is the normal
Cauchy-Riemann operator on $N_{\Sigma/C}$, and its adjoint is $\overline
{\partial}^{\ast}$. So the Fredholm regular property of $\Sigma$ is equivalent
to $\ker\overline{\partial}^{\ast}=\left\{  0\right\}  $, i.e. $\lambda
_{\overline{\partial}^{\ast}}$ $>0$.
\end{proof}

\bigskip

$L^{2}$\textbf{ estimate}

The operator $\mathcal{D=D}_{\pm}:W_{\pm}^{1,2}\left(  \text{A}_{\varepsilon
},\mathbb{S}\right)  \rightarrow L^{2}\left(  \text{A}_{\varepsilon
},\mathbb{S}\right)  $ is self-adjoint by the boundary condition, thus
coker$\mathcal{D}_{-}=\ker\mathcal{D}_{+}$. By the Rayleigh quotient method,
we know
\begin{equation}
\lambda_{\mathcal{D}_{\pm}}:=\inf_{0\neq V\in W_{\pm}^{1,2}\left(
\mathtt{A}_{\varepsilon}\right)  }\frac{\left\Vert \mathcal{D}V\right\Vert
_{L^{2}\left(  \mathtt{A}_{\varepsilon}\right)  }^{2}}{\left\Vert V\right\Vert
_{L^{2}\left(  \mathtt{A}_{\varepsilon}\right)  }^{2}} \label{Rayleigh}%
\end{equation}
is the first eigenvalue of the Laplacian $\mathcal{D}_{\mp}\mathcal{D}_{\pm}$.

\begin{theorem}
\label{1ev}\textbf{(}First eigenvalue estimate\textbf{) }For $\mathcal{D}%
_{\pm}:W_{\pm}^{1,2}\left(  \mathtt{A}_{\varepsilon},\mathbb{S}\right)
\rightarrow L^{2}\left(  \mathtt{A}_{\varepsilon},\mathbb{S}\right)  $, we
have
\[
\lambda_{\mathcal{D}_{-}}\geq\min\left\{  \lambda_{\overline{\partial}}%
,\frac{2}{\varepsilon^{2}}\right\}  \text{ and }\lambda_{\mathcal{D}_{+}}%
\geq\min\left\{  \lambda_{\overline{\partial}^{\ast}},\frac{2}{\varepsilon
^{2}}\right\}  .
\]

\end{theorem}

\begin{remark}
Theorem \ref{1ev} enables us to control the first eigenvalue of
$\mathcal{D}_{\pm}$ on $3$-dimensional $\mathtt{A}_{\varepsilon}$ by that of
$\bar{\partial}$ on $2$-dimensional $\Sigma$, when $\varepsilon$ is small. The
control is due to the boundary condition of $\mathcal{D}_{\pm}$, as will be
clear from the following proof.
\end{remark}

\begin{proof}
We prove the estimate for $\lambda_{\mathcal{D}_{-}}$ ($\lambda_{\mathcal{D}%
_{+}}$ is similar). By the boundary condition $v|_{\partial\mathtt{A}%
_{\varepsilon}}=0$ we have
\begin{align*}
\left\langle \mathcal{D}V,\mathcal{D}V\right\rangle _{L^{2}}  &
=\int_{\left[  0,\varepsilon\right]  \times\Sigma}\left(  \left\vert
\frac{\partial V}{\partial x_{1}}\right\vert ^{2}+\left\Vert \overline
{\partial}^{\ast}v\right\Vert ^{2}+\left\Vert \overline{\partial}u\right\Vert
^{2}\right) \\
&  \geq\int_{\left[  0,\varepsilon\right]  \times\Sigma}\left(  \left\Vert
\overline{\partial}u\right\Vert ^{2}+\left\Vert v_{x_{1}}\right\Vert
^{2}\right)  .
\end{align*}
Then use the Rayleigh quotient for $\overline{\partial}^{\ast}\overline
{\partial}$ and notice that $v|_{\partial\mathtt{A}_{\varepsilon}}=0$.
\end{proof}

\begin{corollary}
\label{D-surj}\bigskip If $\Sigma\subset C$ is Fredholm regular, then for
small enough $\varepsilon>0$, $\mathcal{D}:W_{-}^{1,2}\left(  \mathtt{A}%
_{\varepsilon},\mathbb{S}\right)  \rightarrow L^{2}\left(  \mathtt{A}%
_{\varepsilon},\mathbb{S}\right)  $ is surjective.
\end{corollary}

\begin{proof}
By Lemma \ref{lem:lambda}, the Fredholm regular property of $\Sigma$ implies
that $\lambda_{\overline{\partial}^{\ast}}>0$. So by Theorem \ref{1ev},
$\lambda_{\mathcal{D}_{+}}>0$, i.e. $\ker\mathcal{D}_{+}=\left\{  0\right\}
$. By the self-adjoint property of $\mathcal{D}$, coker$\mathcal{D}_{-}%
=\ker\mathcal{D}_{+}=\left\{  0\right\}  $.
\end{proof}

Now, from the definition of $\lambda_{\mathcal{D}_{-}}$ we obtain the $L^{2}$ estimate

\begin{corollary}
\textbf{(}$L^{2}$-estimate\textbf{) }For any $V\in W_{-}^{1,2}\left(
\mathtt{A}_{\varepsilon},\mathbb{S}\right)  $ and small enough $\varepsilon$,
we have
\begin{equation}
\left\Vert V\right\Vert _{W_{-}^{1,2}\left(  \mathtt{A}_{\varepsilon
},\mathbb{S}\right)  }\leq C\left(  \lambda\right)  \left\Vert \mathcal{D}%
V\right\Vert _{L^{2}\left(  \mathtt{A}_{\varepsilon},\mathbb{S}\right)  },
\label{L2_Estimate}%
\end{equation}
where the constant $C\left(  \lambda\right)  $ only depends on $\lambda
_{\overline{\partial}}$.
\end{corollary}

\bigskip In the following we will derive a $C^{0}$-estimate of $V$. For this
purpose let us examine the equation $\mathcal{D}_{-}V=W$ more closely.

For $V=\left(  u,v\right)  ,$ $W=\left(  w_{1},w_{2}\right)  $, the equation
$\mathcal{D}_{-}V=W$ (assuming $h=1$) is equivalent to the following system
\begin{equation}
\left\{
\begin{array}
[c]{ccc}%
u_{x_{1}}-\mathbf{i}\overline{\partial}^{\ast}v & = & w_{1}\\
v_{x_{1}}+\mathbf{i}\overline{\partial}u & = & w_{2}%
\end{array}
\right.  \ \ \ \text{and \ \ }v|_{\partial\mathtt{A}_{\varepsilon}}=0.
\label{CR1}%
\end{equation}
This is similar to the Cauchy-Riemann equation, but the relation between $u$
and $v$ is \emph{weaker}, since $\nabla u$ can only control $\overline
{\partial}^{\ast}v$ and $v_{x_{1}}$, which are only \emph{half} of partial
derivatives of $v$; the same applies to $\nabla v$. This issue makes it more
difficult to obtain the $C^{0}$-estimate of $u$ than in the Cauchy-Riemann
type equations.

$\bigskip$

$C^{0}$\textbf{ estimate}

\bigskip The $C^{0}$-estimate of $V$ is derived from a $W^{1,p}$-estimate of
$V$ and Sobolev embedding. To obtain a uniform estimate for $\left\Vert
\mathcal{D}^{-1}\right\Vert $ in the $L^{p}$ setting, we use the periodic
reflection technique.

\begin{theorem}
\textbf{(}$L^{p}$-estimate, $p>3$) For any $V\in W_{-}^{1,p}\left(
\mathtt{A}_{\varepsilon},\mathbb{S}\right)  ,$ we have%
\begin{equation}
\left\Vert V\right\Vert _{W_{-}^{1,p}\left(  \mathtt{A}_{\varepsilon
},\mathbb{S}\right)  }\leq C_{p}\left(  \lambda\right)  \left\Vert
\mathcal{D}V\right\Vert _{L^{p}\left(  \mathtt{A}_{\varepsilon},\mathbb{S}%
\right)  }. \label{Lp}%
\end{equation}
where the constant $C_{p}\left(  \lambda\right)  $ only depends on
$\lambda_{\overline{\partial}}$, $Vol\left(  \Sigma\right)  $ and $p$.
\end{theorem}

\begin{proof}
Given any $\varepsilon>0$, we choose an integer $k\left(  \varepsilon\right)
$ so that $1/2\leq k\left(  \varepsilon\right)  \varepsilon\leq3/2$. In the
following we will simply write $k\left(  \varepsilon\right)  $ as $k$.
We\textbf{\ }reflect A$_{\varepsilon}$ to A$_{k\varepsilon}$ periodically and
extend $\left(  u,v\right)  $ along the boundaries $\varepsilon\mathbb{Z}$
$\times\Sigma$\thinspace such that
\begin{align*}
v\left(  x,z\right)   &  =\left\{
\begin{array}
[c]{ccc}%
-v\left(  \left(  2j+2\right)  \varepsilon-x,z\right)  &  & \text{ }%
x\in\left[  \left(  2j+1\right)  \varepsilon,\left(  2j+2\right)
\varepsilon\right] \\
v\left(  x-2j\varepsilon,z\right)  &  & x\in\left[  2j\varepsilon,\left(
2j+1\right)  \varepsilon\right]
\end{array}
\right.  ,\\
u\left(  x,z\right)   &  =\left\{
\begin{array}
[c]{ccc}%
u\left(  \left(  2j+2\right)  \varepsilon-x,z\right)  &  & \text{ }x\in\left[
\left(  2j+1\right)  \varepsilon,\left(  2j+2\right)  \varepsilon\right] \\
u\left(  x-2j\varepsilon,z\right)  &  & x\in\left[  2j\varepsilon,\left(
2j+1\right)  \varepsilon\right]
\end{array}
\right.  ,
\end{align*}
(Notice that a $W^{k,p}$ section will remain so after reflections), i.e. we do
\emph{odd extensions} of $v$ and \emph{even extensions} of $u$ along the
boundaries. By $\left(  \ref{CR1}\right)  $, this induces \emph{odd
extensions} of $w_{1}$ and \emph{even extensions} of $w_{2}$ along the
boundaries. Since the shape of A$_{k\varepsilon}$ is uniformly bounded, we
have the elliptic estimate
\[
\tilde{C}\left(  p\right)  \left\Vert V\right\Vert _{W_{-}^{1,p}\left(
\mathtt{A}_{k\varepsilon},\mathbb{S}\right)  }\leq\left\Vert \mathcal{D}%
V\right\Vert _{L^{p}\left(  \mathtt{A}_{2k\varepsilon}\cup\mathtt{A}%
_{-2k\varepsilon},\mathbb{S}\right)  }+\left\Vert V\right\Vert _{L_{-}%
^{p}\left(  \mathtt{A}_{2k\varepsilon}\cup\mathtt{A}_{-2k\varepsilon
},\mathbb{S}\right)  }.
\]
with a uniform constant $\tilde{C}\left(  p\right)  $ for all $\varepsilon$.
Then we use the interpolation inequality
\begin{equation}
\left\Vert V\right\Vert _{L_{-}^{p}\left(  \mathtt{A}_{k\varepsilon}%
\cup\mathtt{A}_{-2k\varepsilon},\mathbb{S}\right)  }\leq C\left(
\delta^{\frac{p}{p-1}}\left\Vert V\right\Vert _{W^{1,p}\left(  \mathtt{A}%
_{k\varepsilon}\cup\mathtt{A}_{-2k\varepsilon},\mathbb{S}\right)  }%
+\delta^{-p}\left\Vert V\right\Vert _{L_{-}^{2}\left(  \mathtt{A}%
_{k\varepsilon}\cup\mathtt{A}_{-2k\varepsilon},\mathbb{S}\right)  }\right)
\text{ } \label{intropolation}%
\end{equation}
to pass from the $L^{2}$ estimate to the $L^{p}$ estimate%
\begin{equation}
\left\Vert V\right\Vert _{W_{-}^{1,p}\left(  \mathtt{A}_{k\varepsilon
},\mathbb{S}\right)  }\leq C_{p}\left(  \lambda\right)  \left\Vert
\mathcal{D}V\right\Vert _{L^{p}\left(  \mathtt{A}_{k\varepsilon}%
,\mathbb{S}\right)  }. \label{Lp_Ake}%
\end{equation}
This is because the boundary condition of $V\in W_{-}^{1,p}\left(
\mathtt{A}_{\varepsilon},\mathbb{S}\right)  $ and $\left(  \ref{L2_Estimate}%
\right)  $ yields
\[
\left\Vert V\right\Vert _{L_{-}^{2}\left(  \mathtt{A}_{k\varepsilon}%
\cup\mathtt{A}_{-2k\varepsilon},\mathbb{S}\right)  }\leq2C\left(
\lambda\right)  \left\Vert \mathcal{D}V\right\Vert _{L^{2}\left(
\mathtt{A}_{k\varepsilon},\mathbb{S}\right)  }\leq C_{p}\left(  \lambda
\right)  \left\Vert \mathcal{D}V\right\Vert _{L^{p}\left(  \mathtt{A}%
_{k\varepsilon},\mathbb{S}\right)  },
\]
where $C_{p}\left(  \lambda\right)  =2C\left(  \lambda\right)  \left(
\frac{3}{2}Vol\left(  \Sigma\right)  \right)  ^{\frac{1}{2}-\frac{1}{p}}$.
Last we obtain the inequality $\left(  \ref{Lp}\right)  $ on $\mathtt{A}%
_{\varepsilon}$ from $\left(  \ref{Lp_Ake}\right)  $ by the periodicity of the
$L^{p}$ integrals on the reflected domains.
\end{proof}

\begin{corollary}
\label{C0}\textbf{(}$C^{0}$ estimate\textbf{): }If for each $z\in\Sigma$ there
exist $x,x^{\prime}\in\left[  0,\varepsilon\right]  $ such that $u\left(
x,z\right)  =0$ and $v\left(  x^{\prime},z\right)  =0$, then%
\begin{equation}
\left\Vert V\right\Vert _{C^{0}\left(  \mathtt{A}_{\varepsilon},\mathbb{S}%
\right)  }\leq C\varepsilon^{1-\frac{3}{p}}\left\Vert \mathcal{D}V\right\Vert
_{C^{0}\left(  \mathtt{A}_{\varepsilon},\mathbb{S}\right)  }.
\label{C0-estimate}%
\end{equation}

\end{corollary}

\begin{proof}
We have%
\begin{align*}
&  \left\Vert V\right\Vert _{C^{0}\left(  \mathtt{A}_{k\varepsilon}%
,\mathbb{S}\right)  }\overset{\text{Schauder}}{\leq}C\varepsilon^{1-\frac
{3}{p}}\left\Vert V\right\Vert _{C^{1-\frac{3}{p}}\left(  \mathtt{A}%
_{k\varepsilon},\mathbb{S}\right)  }\overset{\text{Sobolev}}{\leq}%
C\varepsilon^{1-\frac{3}{p}}\left\Vert V\right\Vert _{W^{1,p}\left(
\mathtt{A}_{k\varepsilon},\mathbb{S}\right)  }\\
&  \leq C\varepsilon^{1-\frac{3}{p}}\left\Vert \mathcal{D}V\right\Vert
_{L^{p}\left(  \mathtt{A}_{k\varepsilon},\mathbb{S}\right)  }\text{ (by
}\left(  \ref{Lp}\right)  \text{)}\leq C\varepsilon^{1-\frac{3}{p}}\left\Vert
\mathcal{D}V\right\Vert _{C^{0}\left(  \mathtt{A}_{k\varepsilon}%
,\mathbb{S}\right)  },
\end{align*}
where in the first inequality we have used the condition $u\left(  x,z\right)
=0=$ $v\left(  x^{\prime},z\right)  $ and the definition of the Schauder
$C^{1-\frac{3}{p}}$ norm. Then we notice the periodicity on $\mathtt{A}%
_{k\varepsilon}$ so we get $\left(  \ref{C0-estimate}\right)  $.
\end{proof}

$\bigskip$

$C^{1,\alpha}$\textbf{-estimate}

We can change $\mathcal{D}_{-}V=W$ $\left(  \ref{CR1}\right)  $ into a system
of second order elliptic equations%
\begin{equation}
\left\{
\begin{array}
[c]{c}%
u_{x_{1}x_{1}}-\overline{\partial}^{\ast}\overline{\partial}u=\mathbf{i}%
\overline{\partial}^{\ast}w_{2}+\partial_{x_{1}}w_{1}\\
v_{x_{1}x_{1}}-\overline{\partial}\text{ }\overline{\partial}^{\ast
}v=-\mathbf{i}\overline{\partial}w_{1}+\partial_{x_{1}}w_{2}%
\end{array}
\right.  \text{\ and \ \ }v|_{\partial\mathtt{A}_{\varepsilon}}=0,
\label{Dirichlet-uv}%
\end{equation}
where $V=\left(  u,v\right)  $ and $W=\left(  w_{1},w_{2}\right)  $.

To get the uniform estimate of $\left\Vert \mathcal{D}^{-1}\right\Vert $ in
the $C^{1,\alpha}$ setting, we can not just rely on periodic reflections,
because if $w_{1}\neq0$ then the extension of $\mathcal{D}V$ is no longer
continuous on $\mathtt{A}_{k\varepsilon}$, not to mention in $C^{\alpha
}\left(  \mathtt{A}_{k\varepsilon},\mathbb{S}\right)  $.

\begin{theorem}
\textbf{(}$C^{1,\alpha}$ estimate) For any $V\in C_{-}^{1,\alpha}\left(
\mathtt{A}_{\varepsilon},\mathbb{S}\right)  $, we have%
\[
C\varepsilon^{\frac{3}{p}+2\alpha}\left\Vert V\right\Vert _{C_{-}^{1,\alpha
}\left(  \mathtt{A}_{\varepsilon},\mathbb{S}\right)  }\leq\left\Vert
\mathcal{D}V\right\Vert _{C^{\alpha}\left(  \mathtt{A}_{\varepsilon
},\mathbb{S}\right)  }.
\]

\end{theorem}

\begin{proof}
We can decouple $W=\left(  w_{1},w_{2}\right)  $ into dealing with $\left(
0,w_{2}\right)  $ and $\left(  w_{1},0\right)  $ cases separately because
$\mathcal{D}$ is linear and \emph{surjective} from Corollary \ref{D-surj}.

When $w_{1}=0$, along $\partial$A$_{\varepsilon}$ we have $u_{x_{1}}%
=\overline{\partial}^{\ast}v=0$ (for $v|_{\partial\mathtt{A}_{\varepsilon}}%
=0$)$\,$, so we can reflect $u$ by even extension and it is still in
$C^{1,\alpha}$. We can extend $W$ in $C^{\alpha}\left(  \mathtt{A}%
_{\varepsilon},\mathbb{S}\right)  $ as well, for $w_{1}=0$. In this case we
can \emph{restore uniform ellipticity} in the Schauder setting by periodic reflection.

We deal with the more difficult case when $w_{1}\neq0$ but $w_{2}=0$. The
$v$\ component is easier, for it satisfies the Dirichlet boundary condition
$v|_{\partial\text{A}_{\varepsilon}}=0$. Standard Schauder estimate on half
balls (\cite{GT}) implies
\begin{equation}
C\varepsilon^{1+\alpha}\left\Vert v\right\Vert _{C_{-}^{1,\alpha}\left(
\mathtt{A}_{\varepsilon},\mathbb{S}\right)  }\leq\varepsilon\left\Vert
w_{1}\right\Vert _{C^{\alpha}\left(  \mathtt{A}_{\varepsilon},\mathbb{S}%
^{+}\right)  }+\left\Vert v\right\Vert _{C_{-}^{0}\left(  \mathtt{A}%
_{\varepsilon},\mathbb{S}\right)  }. \label{c-alpha}%
\end{equation}
Plugging the\emph{\ }$C^{0}$\emph{\ }estimate $\left(  \ref{C0-estimate}%
\right)  $ in above inequality we get
\[
\left\Vert v\right\Vert _{C_{-}^{1,\alpha}\left(  \mathtt{A}_{\varepsilon
},\mathbb{S}\right)  }\leq C\varepsilon^{-\left(  \frac{3}{p}+\alpha\right)
}\left\Vert W\right\Vert _{C^{\alpha}\left(  \mathtt{A}_{\varepsilon
},\mathbb{S}\right)  }.
\]

The $u$\ component is much harder, since $u|_{\partial\text{A}_{\varepsilon}%
}\neq0$ in general. We carry out the following steps:

(a) \ We \emph{homogenize }$u$ by introducing
\[
\widetilde{u}=u-\rho\left(  \frac{x_{1}}{\varepsilon}\right)  u\left(
\varepsilon,z\right)  -\left(  1-\rho\left(  \frac{x_{1}}{\varepsilon}\right)
\right)  u\left(  0,z\right)  ,
\]
where $\rho:\left[  0,1\right]  \rightarrow\left[  0,1\right]  ,\rho\left(
0\right)  =0,\rho\left(  1\right)  =1$ is a smooth cut-off function such that
$\left\Vert \rho\right\Vert _{C^{2,\alpha}\left[  0,1\right]  }\leq C\,$. Then
we get%
\[
\widetilde{u}_{x_{1}x_{1}}-\partial^{+}\partial^{-}\widetilde{u}%
=\partial_{x_{1}}\widetilde{w_{1}}+g\ \ \ \text{and\ }\ \widetilde
{u}|_{\partial\mathtt{A}_{\varepsilon}}=0,
\]
where the functions $\widetilde{w_{1}}$ and $g$ involve $w_{1}$ and the
boundary values of $u$. Applying elliptic regularity for this Dirichlet
problem of $\widetilde{u}$ and then writing back in $u$, we have
\begin{equation}
\varepsilon\left\Vert w_{1}\right\Vert _{C^{\alpha}}+\varepsilon^{1+\alpha
}\left\Vert u\right\Vert _{C_{-}^{0}}+\left\Vert v\right\Vert _{C_{-}^{0}}\geq
C\varepsilon^{1+2\alpha}\left\Vert u\right\Vert _{C_{-}^{1,\alpha}}.
\label{mixed-elliptic}%
\end{equation}
This is not quite easy, since we can not use $\nabla v$ to completely control
$\nabla u$ from $\left(  \ref{CR1}\right)  $. This was done by carefully
analyzing the $\varepsilon$-orders behaviors of the derivatives of the cut-off
function $\rho\left(  \frac{x_{1}}{\varepsilon}\right)  $ and elliptic
estimates of $\overline{\partial}$ and $\overline{\partial}^{\ast}$ on slices
$\left\{  x\right\}  \times\Sigma$.

(b) We show on each segment $\left[  0,\varepsilon\right]  \times\left\{
z\right\}  $\emph{\ there exists }$x$ \emph{with }$u\left(  x,z\right)  =0$.
Consider the \emph{average}$\ $of $u$ along each segment $\left[
0,\varepsilon\right]  \times\left\{  z\right\}  $:
\[
\overline{u}\left(  z\right)  =\int_{0}^{\varepsilon}u\left(  x_{1},z\right)
dx_{1}.
\]
From $\left(  \ref{CR1}\right)  $ we have (note $w_{2}=0$)
\[
\mathbf{i}\overline{\partial}\overline{u}\left(  z\right)  =-\int
_{0}^{\varepsilon}\partial_{x_{1}}v\left(  x_{1},z\right)  dx_{1}=v\left(
0,z\right)  -v\left(  \varepsilon,z\right)  =0.
\]
But from the assumption that $\overline{\partial}^{\ast}\overline{\partial}$
has trivial kernel on $\Sigma$ we get that
\[
\overline{u}\left(  z\right)  \equiv0.
\]
Then by the intermediate value theorem there must exist some $x\in\left[
0,\varepsilon\right]  $ such that $u\left(  x,z\right)  =0$. So by Corollary
\ref{C0} we get the desired $C^{0}$ estimate of $u$. Similarly for $v$.

(c) Putting the $C^{0}$ estimate of $V=\left(  u,v\right)  $ back to $\left(
\ref{mixed-elliptic}\right)  $, we finally get
\[
\left\Vert V\right\Vert _{C_{-}^{1,\alpha}\left(  \mathtt{A}_{\varepsilon
},\mathbb{S}\right)  }\leq C\varepsilon^{-\left(  \frac{3}{p}+2\alpha\right)
}\left\Vert DV\right\Vert _{C^{\alpha}\left(  \mathtt{A}_{\varepsilon
},\mathbb{S}\right)  }.
\]
Theorem \ref{e-inverse-bound} is proved.
\end{proof}

\subsubsection{Geometry of $J_{n}$-holomorphic curves in $G_{2}$ manifolds:
agreement of two Dirac operators\label{2Dirac}}

Recall given a family $\mathcal{C}=\cup_{0\leq t\leq\varepsilon}C_{t}$ of
coassociative manifolds $C_{t}$ in $M$, the (nonvanishing) deformation vector
field $n:=\frac{dC_{t}}{dt}|_{t=0}$ defines an almost complex structure
$J_{n}$ on $C=C_{0}$. For any $J_{n}$-holomorphic curve $\Sigma\subset C$, the
$G_{2}$-structure on $M$ gives a natural identification between $N_{\Sigma/C}$
and $N_{\mathcal{C}/M}|_{\Sigma}$ (\cite{GayetWitt}). Further calculations
\cite{LWZ1} established the following close relation between the
\textquotedblleft intrinsic\textquotedblright\ and \textquotedblleft
extrinsic\textquotedblright\ Dirac operators.

\begin{proposition}
\label{Dirac-bdl} (Proposition 16,17, \cite{LWZ1}) Let $\Sigma\subset
C\subset\mathcal{C}\subset M$ be as above. We have an orthogonal
decomposition
\[
TM|_{\Sigma}=T\Sigma\oplus N_{\Sigma/C}\oplus N_{\mathcal{C}/M}|_{\Sigma
}\oplus N_{C/\mathcal{C}}|_{\Sigma},
\]

\begin{enumerate}
\item For $L=T\Sigma$, $N_{\Sigma/C}$ or $N_{\mathcal{C}/M}|_{\Sigma}$, the
induced connection $\nabla^{L}$ from the Levi-Civita connection on $M$ is
Hermitian, i.e. $\nabla^{L}J_{n}=0$.

\item The spinor bundle $\mathbb{S}_{\Sigma}$ over $\Sigma$ is identified with
$N_{\Sigma/C}\oplus N_{\mathcal{C}/M}|_{\Sigma}$ in such a way that the
Clifford multiplication is given by the $G_{2}$ multiplication $\times$ and
the spinor connection equals $\nabla^{N_{\Sigma/C}\oplus N_{\mathcal{C}%
/M}|_{\Sigma}}$.

\item The Dirac operator on $\mathbb{S}_{\Sigma}=N_{\Sigma/C}\oplus
N_{\mathcal{C}/M}|_{\Sigma}$ agrees with the Dolbeault operator on
$N_{\Sigma/C}\oplus\wedge_{\mathbb{C}}^{0,1}\left(  N_{\Sigma/C}\right)  $.
\end{enumerate}
\end{proposition}

The agreement of the two Dirac type operators on $\Sigma$ is the geometric
reason that the Fredholm regularity property of the \textquotedblleft
intrinsic\textquotedblright\ Cauchy-Riemann operator on $N_{\Sigma/C}$ gives
control of the \textquotedblleft extrinsic\textquotedblright\ linearized
instanton operator.

\subsubsection{Comparison of $\mathcal{D}$ and linearized instanton equation
\label{exp-like}}

\bigskip When we move from $\Sigma$ to the interior of the almost instanton
$A_{\varepsilon}^{\prime}$, the nice agreement of the two Dirac operators no
longer holds. We need to control the difference between the Dirac operator
$\mathcal{D}$ on $\mathbb{S}$ and the linearized instanton operator
$F_{\varepsilon}^{\prime}\left(  0\right)  $ on $N_{\mathtt{A}_{\varepsilon
}^{\prime}/M}$. In order to compare them, we need a \textit{good}
identification between $\mathbb{S}$ and $N_{\mathtt{A}_{\varepsilon}^{\prime
}/M}$.

For this purpose we defined an \emph{exponential-like map} $\widetilde{\exp}:$
$\mathbb{S\rightarrow}M$ such that $\widetilde{\exp}\left(  \mathtt{A}%
_{\varepsilon}\right)  =\mathtt{A}_{\varepsilon}^{\prime}$ and its
differential $d\widetilde{\exp}|_{\mathtt{A}_{\varepsilon}}$ has the following
properties on $\left\{  0\right\}  \times\Sigma$ (see Appendix of \cite{LWZ1}):

\begin{enumerate}
\item On fiber directions of $\mathbb{S}$, $d\widetilde{\exp}|_{\left\{
0\right\}  \times\Sigma}=\left(  id,f\right)  :N_{\Sigma/C}\oplus
\wedge_{\mathbb{C}}^{0,1}\left(  N_{\Sigma/C}\right)  \rightarrow N_{\Sigma
/C}\oplus N_{\mathcal{C}/M}|_{\Sigma},$ where $f$ is the isomorphism between
the second summands given in Lemma 3.2 of \cite{GayetWitt};

\item On base directions of $\mathbb{S}$, $d\widetilde{\exp}|_{\left\{
0\right\}  \times\Sigma}=id:T\Sigma\rightarrow T\Sigma$ and $d\widetilde{\exp
}|_{\left\{  0\right\}  \times\Sigma}:\frac{\partial}{\partial x_{1}%
}\rightarrow n\left(  z\right)  ;$

\item On the boundary $\left\{  0, \varepsilon\right\}  \times\Sigma$ of $\mathtt{A}%
_{\varepsilon}$, $\widetilde{\exp}|_{\left\{  0\right\}  \times\Sigma}\left(
\mathbb{S}^{+}\oplus0\right)  \subset C_{0}$, $\widetilde{\exp}|_{\left\{ \varepsilon \right\}  \times\Sigma}\left(
\mathbb{S}^{+}\oplus0\right)  \subset C_{\varepsilon}$.
\end{enumerate}

Using $d\widetilde{\exp}|_{\mathtt{A}_{\varepsilon}}$, we can relate
the\emph{\ spin bundle} $\mathbb{S\rightarrow}\mathtt{A}_{\varepsilon}$ to
the\emph{\ normal bundle} $N_{\mathtt{A}_{\varepsilon}^{\prime}/M}$
$\rightarrow\mathtt{A}_{\varepsilon}^{\prime}$. Because of the compatibilities
given in 1-3 above, using Proposition \ref{Dirac-bdl} we have the following
comparison result.

\begin{proposition}
\label{compare-linear-model}\bigskip\ (Proposition 18, \cite{LWZ1}) For any
$V_{1}\in C^{1,\alpha}\left(  A_{\varepsilon},\mathbb{S}\right)  $, we have%
\[
\left\Vert F_{\varepsilon}^{\prime}\left(  0\right)  \left(  d\widetilde{\exp
}\cdot V_{1}\right)  -\left(  d\widetilde{\exp}\right)  \circ\mathcal{D}%
V_{1}\right\Vert _{C^{\alpha}(A_{\varepsilon}^{\prime},N_{A_{\varepsilon
}^{\prime}/M})}\leq C\varepsilon^{1-\alpha}\left\Vert V_{1}\right\Vert
_{C^{1,\alpha}\left(  A_{\varepsilon},\mathbb{S}\right)  }.
\]

\end{proposition}

Then we get the uniform inverse estimate of $F_{\varepsilon}^{\prime}\left(
0\right)  $ from $\left\Vert \mathcal{D}^{-1}\right\Vert $ by the following
diagram, using that $d\widetilde{\exp}$ and $\left(  d\widetilde{\exp}\right)
^{-1}$ are both smooth with a uniform $C^{2}$ bound for all $\varepsilon$:%

\[%
\begin{tabular}
[c]{lll}%
$C_{-}^{1,\alpha}\left(  \mathtt{A}_{\varepsilon},\mathbb{S}\right)  $ &
$\overset{\mathcal{D}}{\longrightarrow}$ & $C^{\alpha}\left(  \mathtt{A}%
_{\varepsilon},\mathbb{S}\right)  $\\
$d\widetilde{\exp}\downarrow$ &  & $\ \downarrow d\widetilde{\exp}$\\
$C_{-}^{1,\alpha}\left(  N_{\mathtt{A}_{\varepsilon}^{\prime}/M}\right)  $ &
$\overset{F_{\varepsilon}^{\prime}\left(  0\right)  }{\longrightarrow}$ &
$C^{\alpha}\left(  N_{\mathtt{A}_{\varepsilon}^{\prime}/M}\right)  $%
\end{tabular}
\ .
\]

\subsubsection{Quadratic estimate\label{subsec:quadratic}}

Using any local frame field $\left\{  W_{\alpha}\right\}  _{\alpha=1}^{7}$, we
could compare the linearizations of $F\left(  V\right)  $ at two different
almost instantons, up to a curvature term $B$ as follows:%

\begin{align}
&  F^{\prime}\left(  V_{0}\right)  |_{A_{\varepsilon\left(  0\right)  }%
}V\left(  p\right) \nonumber\\
&  =\left(  \left(  \exp V_{0}\right)  ^{\ast}\otimes T_{V_{0}}\right)
\circ\left[  F^{\prime}\left(  0\right)  |_{A_{\varepsilon}\left(
V_{0}\right)  }V_{1}\left(  q\right)  \right]  +B^{\alpha}\left(
V_{0},V\right)  W_{\alpha}\left(  p\right)  , \label{lin-functorial}%
\end{align}
where $V_{1}=\left(  d\exp V_{0}\right)  |_{A_{\varepsilon}\left(  0\right)
}V$. This means the derivative $F^{\prime}\left(  V_{0}\right)  $ on
$A_{\varepsilon}\left(  0\right)  $ can be expressed by the derivative
$F^{\prime}\left(  0\right)  $ on $A_{\varepsilon}\left(  V_{0}\right)  $ via
the transform $\left(  \exp V_{0}\right)  ^{\ast}\otimes T_{V_{0}}$, up to the
curvature term $B^{\alpha}\left(  V_{0},V\right)  W_{\alpha}\left(  p\right)
$. Using $\left(  \ref{lin-functorial}\right)  $ we get%
\[
F^{\prime}\left(  V_{0}\right)  V\left(  p\right)  -F^{\prime}\left(
0\right)  V\left(  p\right)  =\left(  I\right)  +\left(  II\right)
\]
with
\begin{align*}
\left(  I\right)   &  =\left[  \left(  \exp V_{0}\right)  ^{\ast}d\left(
i_{V_{1}}\omega^{\alpha}\right)  -d\left(  i_{V}\omega^{\alpha}\right)
\right]  |_{A_{\varepsilon\left(  0\right)  }}\otimes W_{\alpha}\left(
p\right)  ,\\
\left(  II\right)   &  =\left(  \left(  \exp V_{0}\right)  ^{\ast}\otimes
T_{V_{0}}\right)  \circ\left[  i_{V_{1}}d\omega^{\alpha}\otimes W_{\alpha
}+\omega^{\alpha}\otimes\nabla_{V_{1}}W_{\alpha}\right]  \left(  p\right)
+B^{\alpha}\left(  V_{0},V\right)  W_{\alpha}\left(  p\right)  .
\end{align*}
Here $\left(  I\right)  $ consists of 1$^{st}$ order terms whose $C^{\alpha}%
$-norm are bounded by $\left\Vert V_{0}\right\Vert _{C^{1,\alpha}}\left\Vert
V\right\Vert _{C^{1,\alpha}}$ because $d$, the pull back operator, the
parallel transport and various exponential maps are Frechet smooth with
respect to variations of $V$. $\left(  II\right)  $ consists of 0$^{th}$ order
terms and it is easier to bound. Thus we get

\begin{proposition}
\label{Quadratic} (Quadratic estimate) For any $V_{0},V\in C^{1,\alpha}\left(
A_{\varepsilon}^{\prime},N_{A_{\varepsilon}^{\prime}/M}\right)  $ with small
$\left\Vert V_{0}\right\Vert _{C^{1,\alpha}}$, we have%
\begin{equation}
\left\Vert F_{\varepsilon}^{\prime}\left(  V_{0}\right)  V-F_{\varepsilon
}^{\prime}\left(  0\right)  V\right\Vert _{C^{\alpha}}\leq C\left\Vert
V_{0}\right\Vert _{C^{1,\alpha}}\left\Vert V\right\Vert _{C^{1,\alpha},}
\label{Quadratic-estimate}%
\end{equation}
where the constant $C$ is independent on $\varepsilon$.
\end{proposition}

\begin{remark}
\label{WhySchauder}We also have some pointwise estimates tied to the feature
that $\tau$ is a cubic-form. In \cite{LWZ1} we derived that
\begin{align*}
&  \left\vert F^{\prime}\left(  V_{0}\right)  V-F^{\prime}\left(  0\right)
V\right\vert \left(  p\right) \\
&  \leq\left[  C_{7}\left(  \left\vert d\varphi\right\vert +\left\vert \nabla
V_{0}\right\vert \right)  ^{3}\left(  \left\vert V_{0}\right\vert \left\vert
V\right\vert +\left\vert \nabla V_{0}\right\vert \left\vert V\right\vert
+\left\vert V_{0}\right\vert \left\vert \nabla V\right\vert \right)
+C_{5}\left\vert V_{0}\right\vert \left\vert V\right\vert \right]  \left(
p\right)  ,
\end{align*}
where $\varphi:A_{\varepsilon}\rightarrow M$ is the embedding we used to
define $A_{\varepsilon}^{\prime}$. Because of the cubic terms, the following
quadratic estimate needed in the implicit function theorem in $W^{1,p}$
setting
\[
\left\Vert F^{\prime}\left(  V_{0}\right)  V-F^{\prime}\left(  0\right)
V\right\Vert _{L^{p}}\leq C\left\Vert V_{0}\right\Vert _{W^{1,p}}\left\Vert
V\right\Vert _{W^{1,p}.}%
\]
appears unavailable, for $\left(  \left\vert \nabla V_{0}\right\vert
^{3}\right)  ^{p}$ $\notin L^{p}$ in general. In contrast, the Cauchy-Riemann
operator of $J$-holomorphic curves is more linear in the $L^{p}$ setting: it
has the quadratic estimate (Proposition 3.5.3, \cite{MS})%
\[
\left\Vert F^{\prime}\left(  V_{0}\right)  V-F^{\prime}\left(  0\right)
V\right\Vert _{L^{p}\left(  \Sigma\right)  }\leq C\left\Vert V_{0}\right\Vert
_{W^{1,p}\left(  \Sigma\right)  }\left\Vert V\right\Vert _{W^{1,p}\left(
\Sigma\right)  .}%
\]
This is one of the key reasons that we use the Schauder setting.
\end{remark}

By our construction of the almost instanton $A_{\varepsilon}^{\prime}%
=\varphi\left(  A_{\varepsilon}\right)  $ and the smoothness of $\varphi$, it
is not hard to get the error estimate%
\[
\left\Vert F_{\varepsilon}\left(  0\right)  \right\Vert _{_{C^{\alpha}\left(
A_{\varepsilon}^{\prime},N_{A_{\varepsilon}^{\prime}/M}\right)  }}\leq
C\varepsilon^{1-\alpha}\text{. }%
\]
Combining $\left(  \ref{e-right-inverse-estimate}\right)  $, $\left(
\ref{Quadratic-estimate}\right)  $, by the implicit function theorem (e.g.
Proposition A.3.4, \cite{MS}) we can solve $F_{\varepsilon}\left(  V\right)
=0$, thus our main theorem \ref{correspondence} is proved. QED.

\section{Applications and further discussions\label{application}}

\subsection{New examples of instantons\label{new-instanton}}

Our main theorem can be used to construct new examples of instantons.

Let $X$ be a Calabi-Yau threefold containing a complex surface $S\subset X$
which contains a smooth curve $\Sigma\subset S$ which satisfy (i)
$H^{0}\left(  S,K_{S}\right)  \neq0$ (i.e. $p_{g}\left(  S\right)  \neq0$) and
(ii) $H^{0}\left(  \Sigma,K_{S}|_{\Sigma}\right)  =0$.\ Condition (i) implies
that $S$ can be deformed inside $X$ and condition (ii) is equivalent to
$H^{1}\left(  \Sigma,N_{\Sigma/S}\right)  =0$, namely $\Sigma\subset S$ is
Fredholm regular.

Let $\left\{  S_{t}\right\}  _{0\leq t\leq\varepsilon}$ be a smooth family of
deformations of $S$ inside $X$ and let $v=\frac{dS_{t}}{dt}|_{t=0}$ be the
normal vector field on $S=S_{0}$. By our assumption $v$ is nontrivial, and
after possible rescaling of the parameter $t$ for the family $S_{t}$, we may
assume $v$ is very small. Let $M=X\times S^{1}$ be the $G_{2}$ manifold as in
Example \ref{Calabi-Yau-G2}, and $C_{t}:=S_{t}\times\left\{  t\right\}
\subset M$ ($0\leq t\leq\varepsilon$) be the family of coassociative
submanifolds. They are disjoint since their second components $t$ are
different. Let $n_{0}:=\left(  0,\frac{\partial}{\partial\theta}\right)  $ be
the normal vector field on $C_{0}=S\times\left\{  0\right\}  $, then the
original complex structure $J_{0}$ on $S$ is induced from $n_{0}$, i.e.
$J_{0}=\frac{n_{0}}{\left\vert n_{0}\right\vert }\times$.

Let $n=\frac{dC_{t}}{dt}|_{t=0}=\left(  v,\frac{\partial}{\partial\theta
}\right)  $ be the other normal vector field on $C_{0}$. Then $n$ is
nonvanishing for $\frac{\partial}{\partial\theta}$ is nonvanishing on $S^{1}$.
The almost complex structure $J_{n}=\frac{n}{\left\vert n\right\vert }\times$
is close to but not equal to the original complex structure $J_{0}$ on $S$,
because $n$ is close to but not equal to $n_{0}$.

There must exist a Fredholm regular $J_{n}$-holomorphic curve $\Sigma
_{n}\subset$ $S$ near the original $J_{0}$-holomorphic curve $\Sigma$, because
$\Sigma\subset S$ is Fredholm regular and will persist after small
perturbations of $J_{0}$ on $S$.

Applying our main theorem to $\Sigma_{n}\subset C_{0}$, we get an instanton
$A\subset M$ with boundaries on $C_{0}\cup C_{\varepsilon}$. It is \emph{not}
the trivial instanton $\Sigma\times\left[  0,\varepsilon\right]  $, which has
upper boundary lying on $S\times\left\{  \varepsilon\right\}  $, not on
$C_{\varepsilon}=S\left(  \varepsilon\right)  \times\left\{  \varepsilon
\right\}  $.

\subsection{Further remarks \label{discussion}}

A few remarks of our main theorem are in order: First, counting such thin
instantons is basically a problem in four manifold theory because of Bryant's
result \cite{Bryant} which says that the zero section $C$ in $\Lambda_{+}%
^{2}\left(  C\right)  $ is always a coassociative submanifold for an
incomplete $G_{2}$-metric on its neighborhood provided that the bundle
$\Lambda_{+}^{2}\left(  C\right)  $ is topologically trivial.

Second, when the normal vector field $n=\frac{dC_{t}}{dt}|_{t=0}$ has zeros,
our main Theorem \ref{correspondence} should still hold true (work in progress
\cite{LWZ2}). However it would require a possible change of the Fredholm
set-up of the current method and a good understanding of the Seiberg-Witten
theory on any four manifold with a \emph{degenerated symplectic form} as in
Taubes program (\cite{Ta ICM1998}, \cite{Ta SW GW deg}). In the special case
when $\Sigma$ is disjoint from $\left\{  n=0\right\}  $, our theorem
\ref{correspondence} is obviously true as our analysis only involves the local
geometry of $\Sigma$ in $M$.

Third, if we do not restrict to instantons of small volume, then we have to
take into account of \emph{bubbling}{\footnotesize \ }phenomenon as in the
pseudo-holomorphic curves case, and gluing of instantons of big and small
volumes similar to \cite{Oh Zhu} in Floer trajectory case. Nevertheless, one
does not expect bubbling can occur when volumes of instantons are small, thus
they would converge to a $J_{n}$-holomorphic curve in $C$ as $\varepsilon
\rightarrow0$.

Last, we expect our result still holds in the \textit{almost} $G_{2}$ setting,
namely $\Omega$ is only a closed form rather than a parallel form. This is
because our gluing analysis relies mainly on the Fredholm regularity property
of the linearized instanton equation. Such a flexibility could be useful for
finding regular holomorphic curves $\Sigma\subset C$ in order to apply our theorem.

\bigskip

\end{document}